\documentclass[10pt,aps,prb, preprint,reqno]{amsart}
\usepackage{amsmath}
\usepackage{amssymb}
\usepackage{yfonts}
\usepackage{hyperref}
\usepackage{mathrsfs}
\usepackage {latexsym}
\usepackage{graphics}
\usepackage{txfonts}
\usepackage{breqn}
\usepackage{enumitem}
\allowdisplaybreaks

\newtheorem{theorem}{Theorem}[section]
\newtheorem{remark}{Remark}[section]

\newtheorem{proposition}[theorem]{Proposition}

\newcommand{\joinR}{\hspace{-.1em}}
\newcommand{\RomanI}{I}
\newcommand{\RomanII}{\mbox{\RomanI\joinR\RomanI}}
\newcommand{\RomanIII}{\mbox{\RomanI\joinR\RomanII}}
\newcommand{\RomanIV}{\mbox{\RomanI\joinR\RomanV}}
\newcommand{\RomanV}{V}
\newcommand{\RomanVI}{\mbox{\RomanV\joinR\RomanI}}
\newcommand{\RomanVII}{\mbox{\RomanV\joinR\RomanII}}
\newcommand{\RomanVIII}{\mbox{\RomanV\joinR\RomanIII}}

\begin{document}
\title[Two and a half dimensional magnetohydrodynamics system]{Remarks on the global regularity issue of the two and a half dimensional Hall-magnetohydrodynamics system}
 
\date{} 
 
\subjclass[2010]{35Q30; 76W05} 
\keywords{Keywords: Global well-posedness; Hall-magnetohydrodynamics system; Magnetohydrodynamics system; Navier-Stokes equations; Regularity criteria.}
 
\author[Rahman]{Mohammad Mahabubur Rahman}
\address{Department of Mathematics and Statistics, Texas Tech University, Lubbock, Texas, 79409-1042, U.S.A.}
\email{mohammad-mahabu.rahman@ttu.edu}
\author[Yamazaki]{Kazuo Yamazaki}
\address{Department of Mathematics and Statistics, Texas Tech University, Lubbock, Texas, 79409-1042, U.S.A.}
\email{kyamazak@ttu.edu} 

\begin{abstract}
Whether or not the solution to the $2\frac{1}{2}$-dimensional Hall-magnetohydrodynamics system starting from smooth initial data preserves its regularity for all time remains a challenging open problem. Although the research direction on component reduction of regularity criterion for Navier-Stokes equations and magnetohydrodynamics system have caught much attention recently, the Hall term has presented much difficulty. In this manuscript we discover a certain cancellation within the Hall term and obtain various new regularity criterion: first, in terms of a gradient of only the third component of the magnetic field; second, in terms of only the third component of the current density; third, in terms of only the third component of the velocity field; fourth, in terms of only the first and second components of the velocity field. As another consequence of the cancellation that we discovered, we are able to prove the global well-posedness of the $2\frac{1}{2}$-dimensional Hall-magnetohydrodynamics system with hyper-diffusion only for the magnetic field in the horizontal direction; we also obtained an analogous result in the 3-dimensional case via discovery of additional cancellations. These results extend and improve various previous works.
\end{abstract}

\maketitle

\section{Introduction}\label{Introduction}

\subsection{Motivation from physics and mathematics}\label{Motivation from physics and mathematics}

Initiated by Alfv$\acute{\mathrm{e}}$n \cite{A42a} in 1942, the study of magnetohydrodynamics (MHD) concerns the properties of electrically conducting fluids. For example, while fluid turbulence is often investigated through Navier-Stokes (NS) equations, MHD turbulence occurs in laboratory settings such as fusion confinement devices (e.g., reversed field pinch), as well as astrophysical systems (e.g., solar corona) and the conventional system of equations for such a study is that of the MHD. Lighthill in 1960 introduced the Hall-MHD system by adding a Hall term which actually arises upon writing the current density as the sum of the ohmic current and a Hall current that is perpendicular to the magnetic field (see \cite[equation (94)]{L60}). Ever since, the Hall-MHD system has found tremendous attractions due to its applicability: study of the sun (e.g., \cite{C98}); magnetic re-connection (e.g. \cite{DSDCSCM12, HG05}); turbulence (e.g. \cite{MH09}); star formation (e.g. \cite{W04}). In particular, the statistical study of magnetic re-connection events in \cite{DSDCSCM12} by Donato et al. focused on the two and a half dimensional (2$\frac{1}{2}$-D) Hall-MHD system, necessarily because the Hall term decouples from the rest of the system in the 2-D case (see equation \eqref{est 3}). It is well-known that despite much advances made in the mathematical analysis of the Hall-MHD system in the past decades, the global well-posedness of the $2\frac{1}{2}$-D Hall-MHD system remains open: e.g., ``\emph{Contrary to the usual MHD the global well-posedness in the $2\frac{1}{2}$ dimensional Hall-MHD is wide open}'' from \cite[Abstract]{CL14} by Chae and Lee. The purpose of this manuscript is to present some new cancellations within the Hall term (see \eqref{est 19}, \eqref{est 20}, \eqref{new1}, \eqref{new2}, and \eqref{new3}) and obtain new regularity criterion in terms of only a few components of the solution, as well as prove the global well-posedness of the $2\frac{1}{2}$-D Hall-MHD system with hyper-diffusion for magnetic field only in the horizontal direction (see also Theorem \ref{Theorem 2.5} in the 3-D case). In Remark \ref{Remark 2.1} we also list some interesting open problems for future works.

\subsection{Previous works}\label{Previous works}
Let us denote $\partial_{t} \triangleq \frac{\partial}{\partial t}, \partial_{j} \triangleq \frac{\partial}{\partial x_{j}}$ for $j \in \{1,2,3\}$, components of any vector by sub-index such as $x = (x_{1}, \hdots, x_{D})$, and $A \overset{(\cdot)}{\lesssim} B$ to imply the existence of a constant $C \geq 0$ of no dependence on any important parameter such that $A \leq CB$ due to equation $(\cdot)$. The spatial domain $\Omega$ of our interest will be $\mathbb{T}^{D}$ or $\mathbb{R}^{D}$ (see \cite[Remark 1]{ADFL11} concerning the difficulty in the case of a general domain). For clarity let us first present the 3-D Hall-MHD system. With $u: \mathbb{R}_{\geq 0} \times \Omega\mapsto \mathbb{R}^{3}$, $b: \mathbb{R}_{\geq 0} \times \Omega \mapsto \mathbb{R}^{3}$, $\pi: \mathbb{R}_{\geq 0} \times \Omega \mapsto \mathbb{R}$, and $j \triangleq \nabla \times b$, representing respectively velocity field, magnetic field, pressure field, and current density, as well as $\nu \geq 0$, $\eta \geq 0$, and $\epsilon \geq 0$ respectively the viscosity, magnetic diffusivity, and Hall parameter, this system reads  
\begin{subequations}\label{est 0}
\begin{align}
& \partial_{t} u + (u\cdot\nabla) u + \nabla \pi = \nu \Delta u + (b\cdot\nabla) b  \hspace{17mm} t > 0,\label{est 1}\\
& \partial_{t} b + (u\cdot\nabla) b = \eta \Delta b + (b\cdot\nabla) u - \epsilon \nabla \times ( j \times b) \hspace{4mm} t > 0, \label{est 2}\\
& \nabla\cdot u = 0, \nabla \cdot b = 0, 
\end{align}
\end{subequations} 
starting from initial data $(u_{0}, b_{0}) \triangleq (u,b) \rvert_{t=0}$. The case $\epsilon = 0$ reduces \eqref{est 0} to the MHD system, furthermore taking $b \equiv 0$ and assuming $\nu > 0$ transform \eqref{est 0} to the NS equations, and additionally taking $\nu = 0$ leads us to the Euler equations. Throughout the rest of this manuscript, we will assume $\nu, \eta > 0$. It can be immediately seen from \eqref{est 1}-\eqref{est 2} that the 2-D case of $u(t,x) = (u_{1}, u_{2})(t,x_{1}, x_{2})$, $b(t,x) = (b_{1}, b_{2})(t, x_{1}, x_{2})$ result in the decoupling of the Hall term from the rest; for this reason, physicists such as Donato et al. in \cite{DSDCSCM12} have turned to the $2\frac{1}{2}$-D case (see also \cite[Section 2.3.1]{MB02}): 
\begin{equation}\label{est 3} 
u(t,x) = (u_{1}, u_{2}, u_{3})(t, x_{1}, x_{2}) \hspace{1mm} \text{ and } \hspace{1mm} b(t,x) = (b_{1}, b_{2}, b_{3})(t, x_{1}, x_{2}) 
\end{equation}
that solves \eqref{est 0}. 

We now review recent developments on the mathematical theory of the Hall-MHD system. Acheritogaray, Degond, Frouvelle, and Liu \cite{ADFL11} proved the global existence of a weak solution to the 3-D Hall-MHD system in case $\Omega = \mathbb{T}^{3}$ by using a vector calculus identity 
\begin{equation}\label{key}
(\Theta \times \Psi) \cdot \Theta = 0 \hspace{3mm} \forall \hspace{1mm} \Theta, \Psi \in \mathbb{R}^{3} 
\end{equation}
which led to the zero contribution from the Hall term on the energy identity. This pioneering work led to many others, some of which we list and refer readers to their references: \cite{CDL14} on various well-posedness results in $\mathbb{R}^{3}$; \cite{CS13} on temporal decay in $\mathbb{R}^{3}$; \cite{CW15, CW16b} for partial regularity results; \cite{CWW15} for local well-posedness in case the magnetic diffusion $-\Delta b$ is generalized to $(-\Delta)^{\alpha} b$ for any $\alpha > \frac{1}{2}$; \cite{CW16a} for singularity formation in case of zero magnetic diffusion. 

Considering that the global regularity problem of the 3-D and even $2\frac{1}{2}$-D Hall-MHD system remain completely open, one natural approach would be to pursue Prodi-Serrin type regularity criteria (\cite{ESS03, P59, S62}). Such results have been extended to the MHD system, e.g., \cite{HX05, Z05} in terms of only $u$, as well as the Hall-MHD system, e.g., Chae and Lee who obtained various blow-up criterion such as 
\begin{equation}\label{est 10} 
\limsup_{t\nearrow T^{\ast}} (\lVert u(t) \rVert_{H^{m}}^{2} + \lVert b(t) \rVert_{H^{m}}^{2}) = \infty \text{ if and only if } \int_{0}^{T^{\ast}} (\lVert u \rVert_{BMO}^{2} + \lVert \nabla b \rVert_{BMO}^{2}) dt = \infty
\end{equation} 
for $m > \frac{5}{2}$ that is an integer in the 3-D case (see \cite[Theorem 2]{CL14}) while 
\begin{equation}\label{est 4} 
\limsup_{t\nearrow T^{\ast}} (\lVert u(t) \rVert_{H^{m}}^{2} + \lVert b(t) \rVert_{H^{m}}^{2}) = \infty \text{ if and only if } \int_{0}^{T^{\ast}} \lVert j \rVert_{BMO}^{2} dt = \infty 
\end{equation}
for $m > 2$ that is an integer in the $2\frac{1}{2}$-D case (see \cite[Remark 3]{CL14}). Furthermore, \cite[Theorem 1]{CL14} stated a Prodi-Serrin type regularity criteria for the 3-D Hall-MHD system, specifically that $u, \nabla b \in L_{T}^{r}L_{x}^{p}$ for $\frac{3}{p} + \frac{2}{r} \leq 1, p \in (3,\infty]$ (also \cite[Theorem 1.1]{Y15}), and it can be immediately generalized to show that $2\frac{1}{2}$-D Hall-MHD system is globally well-posed as long as 
\begin{equation}\label{est 5} 
\nabla b \in L_{T}^{r} L_{x}^{p} \text{ for } \frac{2}{p} + \frac{2}{r} \leq 1, p \in (2,\infty]. 
\end{equation} 
\begin{remark}\label{Remark 1.1}
The intuition behind \eqref{est 10} concerning why a bound on $\nabla b$ is needed when its counterpart is $u$, as well as why we only need a bound on $\nabla b$ in \eqref{est 5} is because the Hall term $\nabla \times (j\times b)$ is precisely more singular than those in the MHD system by one derivative and the  $2\frac{1}{2}$-D MHD system is known to be globally well-posed (the proof follows the 2-D case verbatim, e.g., \cite{ST83}).
\end{remark}

In the past few decades, a research direction on component reduction of such Prodi-Serrin type regularity criterion have flourished. As far back as 1999, Chae and Choe \cite{CC99} proved a regularity criteria in terms of only two components of the vorticity vector field $\omega \triangleq \nabla \times u$ for the 3-D NS equations; we also refer to \cite{KZ07} on a criteria in terms of $\partial_{3} u$ and \cite{CT08} in terms of $u_{3}$. In particular, let us briefly elaborate on the method of Kukavica and Ziane \cite{KZ06} which inspired many more results on related equations. Denoting a horizontal gradient and a horizontal Laplacian respectively by $\nabla_{h} \triangleq (\partial_{1}, \partial_{2}, 0)$ and $\Delta_{h} \triangleq \sum_{k=1}^{2} \partial_{k}^{2}$, Kukavica and Ziane observed that a solution to the 3-D NS equations satisfies  
\begin{subequations}\label{KZ}
\begin{align} 
&\frac{1}{2} \partial_{t} \lVert \nabla_{h} u \rVert_{L^{2}}^{2} + \nu \lVert \nabla \nabla_{h} u \rVert_{L^{2}}^{2}  \lesssim \int_{\mathbb{R}^{3}} \lvert u_{3} \rvert \lvert \nabla u \rvert \lvert \nabla \nabla_{h} u \rvert dx,\label{est 6}\\
&\frac{1}{2} \partial_{t} \lVert \nabla u \rVert_{L^{2}}^{2} + \nu \lVert \Delta u \rVert_{L^{2}}^{2} \lesssim \int_{\mathbb{R}^{3}} \lvert \nabla_{h} u \rvert \lvert \nabla u \rvert^{2} dx.\label{est 7} 
\end{align}
\end{subequations}
It follows that one can assume a bound on $u_{3}$, apply it in \eqref{est 6} to get a bound on $\nabla_{h} u$, and then use the bound on $\nabla_{h} u$ in \eqref{est 7} to attain the desired bound on $\nabla u$, although this extra step has always prevented one from attaining the criteria at a scaling-invariant level. The interesting part is how one can separate $u_{3}$ in \eqref{est 6}; it turns out that using divergence-free property of $u$, one can compute 
\begin{align}
&\frac{1}{2} \partial_{t} \lVert \nabla_{h} u \rVert_{L^{2}}^{2} + \nu \lVert \nabla \nabla_{h} u \rVert_{L^{2}}^{2} = \int_{\mathbb{R}^{3}}(u\cdot\nabla) u \cdot \Delta_{h} u dx \label{est 8}\\
=& - \sum_{i,l,k=1}^{2}\int_{\mathbb{R}^{3}}\partial_{k}u_{i}\partial_{i}u_{l}\partial_{k}u_{l} dx + \sum_{l,k=1}^{2}\int_{\mathbb{R}^{3}} u_{3} \partial_{k} (\partial_{3}u_{l}\partial_{k} u_{l}) dx + \sum_{i=1}^{3}\sum_{k=1}^{2}\int_{\mathbb{R}^{3}} u_{3} \partial_{i} (\partial_{k} u_{i} \partial_{k} u_{3}) dx. \nonumber
\end{align}
The second and third integrals already have $u_{3}$ separated. For the remaining integral of $\sum_{i,l,k=1}^{2}\int_{\mathbb{R}^{3}}\partial_{k}u_{i}\partial_{i}u_{l}\partial_{k}u_{l} dx$ which do not seem to have any $u_{3}$, strategic couplings and making use of the divergence-free property lead to the desired bound \eqref{est 6}; e.g., the sum of two terms $(i,l,k) = (1,1,2)$ and $(i,l,k) = (2,1,2)$ lead to 
\begin{equation}\label{est 9}
\int_{\mathbb{R}^{3}} \partial_{2} u_{1} \partial_{1} u_{1} \partial_{2} u_{1} + \partial_{2} u_{2} \partial_{2} u_{1}\partial_{2} u_{1} dx = \int_{\mathbb{R}^{3}} (\partial_{2} u_{1})^{2} (-\partial_{3} u_{3}) dx \lesssim \int_{\mathbb{R}^{3}} \lvert u_{3} \rvert \lvert \nabla u \rvert \lvert \nabla \nabla_{h} u \rvert dx.
\end{equation} 
For the MHD system, specifically \eqref{est 0} with $\epsilon = 0$, lengthy computations and strategic couplings of terms in $(u\cdot\nabla)u$ with $(u\cdot\nabla) b$, as well as $(b\cdot\nabla) b$ with $(b\cdot\nabla) u$, led to discovery of various hidden identities (e.g., \cite[Proposition 3.1]{Y14a}). However, such cancellations seemed virtually impossible for the Hall-MHD system; a curl operator makes it one more derivative singular and there is no other nonlinear term of same order to couple together. 

Subsequently, Chemin and Zhang \cite{CZ16} employed anisotropic Littlewood-Paley theory to obtain a regularity criteria for the 3-D NS equations in terms of only $u_{3}$ in a scaling-invariant space, although in Sobolev space $\dot{H}^{\frac{1}{2} + \frac{2}{p}}(\mathbb{R}^{3})$ for $p \in (4,6)$ rather than Lebesgue space $L^{p}(\mathbb{R}^{3})$ space (see also \cite{CZZ17}). This approach of anisotropic Littlewood-Paley theory was also successfully extended to the MHD system in \cite{Y16c} by discovery of some crucial cancellations, e.g., \cite[Equations (21)-(22)]{Y16c} (see also \cite{HZ20, L16}). Again, to the best of the authors' knowledge, an attempt to extend this approach to the Hall-MHD system has not been done due to the difficulty of the Hall term (see also \cite{WWZ20} for recent development on another approach). 

\section{Statement of main results}\label{Statement of main results}
We overcame the challenges described in Section \ref{Previous works} and obtained component reduction results of regularity criteria for the $2\frac{1}{2}$-D Hall-MHD system. For clarity we work on $\mathbb{R}^{2}$ hereafter although all of our works may be instantly extended to the case $x \in \mathbb{T}^{2}$ by straight-forward modifications.  One immediate convenience of working on the $2\frac{1}{2}$-D Hall-MHD system is that $H^{1}(\mathbb{R}^{2})$-bound on $(u,b)$ immediately implies higher regularity. This can be seen from the facts that an $H^{1}(\mathbb{R}^{2})$-bound gives $\int_{0}^{T} \lVert \Delta b \rVert_{L^{2}}^{2} dt < \infty$ from magnetic diffusive term, that $BMO(\mathbb{R}^{d})$ norm may be bounded by $\dot{H}^{\frac{d}{2}}(\mathbb{R}^{d})$ norm (e.g., \cite[Theorem 1.48]{BCD11}), and \eqref{est 4}; this also shows that $H^{1}(\mathbb{R}^{2})$-bound does not suffice in the 3-D case considering \eqref{est 10}. For clarity let us note that in the $2\frac{1}{2}$-D case 
\begin{equation}\label{j}
j = \nabla \times b = 
\begin{pmatrix}
\partial_{2}b_{3}& - \partial_{1} b_{3}& \partial_{1} b_{2} - \partial_{2} b_{1}
\end{pmatrix}^{T}.
\end{equation} 
Our first result is a criteria in terms of $\nabla b_{3}$ or equivalently $j_{1}$ and $j_{2}$; because we get it at the level of $\frac{2}{p} + \frac{2}{r} \leq 1, p \in (2,\infty]$, these are direct improvements of \eqref{est 5}. The heuristic that led us to pursue this result is as follows. On one hand, the $2\frac{1}{2}$-D MHD system is globally well-posed in $H^{m}(\mathbb{R}^{2})$ for any $m > 2$; on the other hand, if $b_{3} \equiv 0$, then the Hall term within the $2\frac{1}{2}$-D Hall-MHD system decouples and the resulting system is globally well-posed. Such intuitions indicate that even if $b_{3}$ does not necessarily vanish but as long as it is sufficiently small or bounded in some norm, the global well-posedness of the $2\frac{1}{2}$-D Hall-MHD system may still hold; the following result answers this question in positive. 

\begin{theorem}\label{Theorem 2.1}
Suppose that $(u_{0}, b_{0}) \in H^{m} (\mathbb{R}^{2}) \times H^{m} (\mathbb{R}^{2})$ where $m > 2$ is an integer and $\nabla\cdot u_{0} = \nabla\cdot b_{0} = 0$. If $(u,b)$ is a corresponding local smooth solution to the $2\frac{1}{2}$-D Hall-MHD system \eqref{est 0} emanating from $(u_{0}, b_{0})$ and 
\begin{equation}
\int_{0}^{T} \lVert \nabla b_{3} \rVert_{L^{p}}^{r} dt \overset{\eqref{j}}{=} \sum_{k=1}^{2} \int_{0}^{T} \lVert j_{k} \rVert_{L^{p}}^{r} dt < \infty \text{ where } \frac{2}{p} + \frac{2}{r}\leq 1, p \in (2,\infty],
\end{equation} 
then for all $t \in [0, T]$ 
\begin{equation}\label{est 27} 
\lVert u (t) \rVert_{H^{m}} + \lVert b(t) \rVert_{H^{m}} < \infty.
\end{equation} 
\end{theorem}

The proof of Theorem \ref{Theorem 2.1} is not a two-step approach as in \cite{KZ06} (recall \eqref{KZ}); we directly work on the $H^{1}(\mathbb{R}^{2})$-estimate, discover crucial cancellations \eqref{est 19} and \eqref{est 20}. Thus, we are able to attain the criteria at the level of $\frac{2}{p} + \frac{2}{r} \leq 1, p \in (2,\infty]$, somewhat similarly to the works of \cite{LRY21, Y13a}. 

Next, having obtained a criteria in terms of $j_{1}$ and $j_{2}$ rather than $j$, a natural question is whether we can reduce to only one component of the current density which would be a further improvement of \eqref{est 4}; that is the content of our next result. 

\begin{theorem}\label{Theorem 2.2}
Suppose that $(u_{0}, b_{0}) \in H^{m} (\mathbb{R}^{2}) \times H^{m} (\mathbb{R}^{2})$ where $m > 2$ is an integer and $\nabla\cdot u_{0} = \nabla\cdot b_{0} = 0$. If $(u,b)$ is a corresponding local smooth solution to the $2\frac{1}{2}$-D Hall-MHD system \eqref{est 0} emanating from $(u_{0}, b_{0})$ and 
\begin{equation}
\int_{0}^{T} \lVert j_{3} \rVert_{L^{p}}^{r} dt  < \infty \text{ where } \frac{2}{p} + \frac{2}{r}\leq 1, p \in (2,\infty],
\end{equation} 
then for all $t \in [0, T]$ \eqref{est 27} holds. 
\end{theorem}
\noindent (Cf. \cite{Y14a} in which a regularity criteria for 3-D MHD system in terms of $u_{3}$ and $j_{3}$ was obtained). 

Next, having obtained a criteria in terms of $j_{1}$ and $j_{2}$ or just $j_{3}$ alone, a natural question is whether we can obtain a similar result in terms of components of velocity or vorticity $\omega = \nabla \times u$. Intuitively, this seems very difficult. As we mentioned, the $2\frac{1}{2}$-D MHD system is globally well-posed and the only difference with the $2\frac{1}{2}$-D Hall-MHD system is the Hall term $\nabla \times (j\times b)$. Upon any energy estimate such as the bounds on $L^{2}(\mathbb{R}^{2}), L^{p}(\mathbb{R}^{2})$ for $p \in (2,\infty]$, $\dot{H}^{1}(\mathbb{R}^{2})$ or $\dot{H}^{2}(\mathbb{R}^{2})$, the Hall term is always multiplied by a magnetic field again; therefore, any bound on $u$ does not seem to offer any help to handle the Hall term. The trick is that, as we will see in \eqref{Hall rewritten}, we may rewrite the Hall term as $\nabla \times ((b\cdot\nabla) b)$ and realize that this term appears naturally in the equation of vorticity $\omega$ (see equation \eqref{vorticity}); therefore, upon any estimate on vorticity, this ``Hall term within the equation of vorticity'' gets multiplied by vorticity $\omega$, giving us a chance to show that a sufficient bound on some components of vorticity may allow one to bound this nonlinear term. The following result answers this question in positive. 

\begin{theorem}\label{Theorem 2.3}
Suppose that $\nu = \eta = \epsilon = 1$. Suppose that $(u_{0}, b_{0}) \in H^{m} (\mathbb{R}^{2}) \times H^{m} (\mathbb{R}^{2})$ where $m > 2$ is an integer and $\nabla\cdot u_{0} = \nabla\cdot b_{0} = 0$.  
\begin{enumerate}
\item If $(u,b)$ is a corresponding local smooth solution to the $2\frac{1}{2}$-D Hall-MHD system \eqref{est 0} emanating from $(u_{0}, b_{0})$ and
\begin{equation}\label{est 52}
\int_{0}^{T} \lVert \Delta (\Delta u_{3}) \rVert_{L^{p}}^{r} dt < \infty \text{ where } \frac{2}{p} + \frac{2}{r} \leq 1, p \in (2,\infty), 
\end{equation} 
or 
\begin{equation}\label{est 53}
\int_{0}^{T} \lVert \Delta (\Delta u_{3}) \rVert_{BMO}^{2} dt < \infty, 
\end{equation} 
then for all $t \in [0, T]$ \eqref{est 27} holds. 
\item If $(u,b)$ is a corresponding local smooth solution to the $2\frac{1}{2}$-D Hall-MHD system \eqref{est 0} emanating from $(u_{0}, b_{0})$ and
\begin{equation}\label{est 61}
\int_{0}^{T} \lVert \Delta (\Delta u_{k})  \rVert_{L^{p_{k}}}^{r_{k}} dt < \infty   \text{ where } \frac{2}{p_{k}} + \frac{2}{r_{k}} \leq 1, p_{k} \in (2,\infty), 
\end{equation} 
or 
\begin{equation}\label{est 62}
\int_{0}^{T} \lVert \Delta (\Delta u_{k}) \rVert_{BMO}^{2} dt < \infty 
\end{equation} 
for both $k \in \{1,2\}$, then for all $t \in [0, T]$ \eqref{est 27} holds. 
\end{enumerate} 
\end{theorem}
After this work was completed, we were informed of the results in \cite{HAHZ16} in which the authors also obtained several regularity criterion of the 3-D Hall-MHD system in terms of $u$ in a high-order norm, e.g., $\omega \in L_{T}^{2} \dot{H}_{x}^{m}$ for $m > \frac{5}{2}$ in \cite[Theorem 1.3]{HAHZ16}. In comparison, while Theorem \ref{Theorem 2.3} is concerned with the $2\frac{1}{2}$-D case, its criterion are on only a few number of components of $u$. 

Our last result in the $2\frac{1}{2}$-D case is about global well-posedness of a certain system of equations with magnetic hyper-diffusion. In \cite[Theorem 2.3]{Y19a}, the global well-posedness of the following system was shown:
\begin{subequations}\label{est 68}
\begin{align}
&\partial_{t} u + (u\cdot\nabla) u + \nabla \pi  = \nu \Delta u + (b\cdot\nabla) b  \hspace{20mm} t > 0, \label{est 63} \\
&\partial_{t} b + (u\cdot\nabla) b + \eta \Lambda^{3} b = (b\cdot\nabla) u - \epsilon \nabla \times (j\times b)  \hspace{5mm} t > 0.  \label{est 64} 
\end{align}
\end{subequations} 
The intuition behind why we need precisely one more derivative in the magnetic diffusion, namely $\Lambda^{3} b$ rather than $-\Delta b$, is similar to the explanation in Remark \ref{Remark 1.1}. Inspired partially by \cite{YJW19} concerning the 3-D NS equations with various powers of fractional diffusion in different directions, we consider the following system:
\begin{subequations}\label{est 67}
\begin{align}
&\partial_{t} u + (u\cdot\nabla) u + \nabla \pi  =\nu \Delta u +  (b\cdot\nabla) b \hspace{37mm} t > 0, \label{est 65} \\
&\partial_{t} b + (u\cdot\nabla) b + \eta_{h} \Lambda^{3} b_{h} + \eta_{v} \Lambda^{2} b_{v} = (b\cdot\nabla) u - \epsilon \nabla \times (j\times b) \hspace{5mm} t > 0,  \label{est 66} 
\end{align}
\end{subequations} 
where
\begin{equation}\label{est 77}
b_{h} \triangleq  
\begin{pmatrix}
b_{1} & b_{2} & 0 
\end{pmatrix}^{T} \text{ and }  b_{v} \triangleq  
\begin{pmatrix}
0 & 0 & b_{3} 
\end{pmatrix}^{T}.
\end{equation} 
Clearly the system \eqref{est 67} has weaker magnetic diffusion in the vertical direction than \eqref{est 68}; in fact, it has same magnetic diffusion strength as the classical Hall-MHD system. Nonetheless, we are able to prove its global well-posedness as a consequence of the proofs of Theorems \ref{Theorem 2.1}-\ref{Theorem 2.3}. 

\begin{theorem}\label{Theorem 2.4}
 Suppose that $(u_{0}, b_{0}) \in H^{m} (\mathbb{R}^{2}) \times H^{m} (\mathbb{R}^{2})$ where $m > 2$ is an integer and $\nabla\cdot u_{0} = \nabla\cdot b_{0} = 0$.  Then there exists a unique solution 
 \begin{equation}
u \in L^{\infty} ((0,\infty); H^{m} (\mathbb{R}^{2})) \cap L^{2} ((0,\infty); H^{m+1}(\mathbb{R}^{2})), b \in L^{\infty} ((0,\infty);H^{m}(\mathbb{R}^{2})) 
\end{equation} 
such that 
\begin{equation}
b_{h} \in L^{2} ((0,\infty); H^{m+ \frac{3}{2}} (\mathbb{R}^{2})), b_{v} \in L^{2} ((0,\infty); H^{m+1}(\mathbb{R}^{2}))
\end{equation} 
to \eqref{est 67}, and $(u,b) \rvert_{t=0} = (u_{0}, b_{0})$. 
\end{theorem}

As we emphasized, the crux of the proofs of Theorems \ref{Theorem 2.1}-\ref{Theorem 2.4} rely on the cancellations \eqref{est 19} and \eqref{est 20}. After obtaining Theorems \ref{Theorem 2.1}-\ref{Theorem 2.4}, the second author Yamazaki shared these results with Prof. Mimi Dai and her comments inspired us to further try to generalize to the 3-D case. We were quite surprised to learn that the cancellations that we discovered, specifically \eqref{est 19} and \eqref{est 20}, are indeed general and can be extended to the 3-D case (see \eqref{new1}, \eqref{new2}, and \eqref{new3}). As we elaborated already, an $H^{1}(\mathbb{R}^{3})$-bound on the 3-D Hall-MHD system does not suffice to lead to higher regularity; thus, at the time of writing this manuscript, it is not clear how to extend Theorems \ref{Theorem 2.1}-\ref{Theorem 2.3} to the 3-D case. On the other hand, via new cancellations \eqref{new1}, \eqref{new2}, and \eqref{new3} we can extend Theorem \ref{Theorem 2.4} to the 3-D case as follows. Under same notations of $b_{h}$ and $b_{v}$ from \eqref{est 77} we consider 
\begin{subequations}\label{est 78}
\begin{align}
&\partial_{t} u + (u\cdot\nabla) u + \nabla \pi + \nu \Lambda^{\frac{5}{2}} u=  (b\cdot\nabla) b,  \hspace{37mm} t > 0, \label{est 79} \\
&\partial_{t} b + (u\cdot\nabla) b + \eta_{h} \Lambda^{\frac{7}{2}} b_{h} + \eta_{v} \Lambda^{\frac{5}{2}} b_{v} = (b\cdot\nabla) u - \epsilon \nabla \times (j\times b),  \hspace{7mm} t > 0.  \label{est 80} 
\end{align}
\end{subequations} 
\begin{theorem}\label{Theorem 2.5}
Suppose that $(u_{0}, b_{0}) \in H^{m} (\mathbb{R}^{3}) \times H^{m}(\mathbb{R}^{3})$ where $m > \frac{5}{2}$ is an integer and $\nabla\cdot u_{0} = \nabla\cdot b_{0} = 0$.  Then there exists a unique solution 
 \begin{equation}
u \in L^{\infty} ((0,\infty); H^{m} (\mathbb{R}^{3})) \cap L^{2} ((0,\infty); H^{m+\frac{5}{4}}(\mathbb{R}^{3})), b \in L^{\infty} ((0,\infty);H^{m}(\mathbb{R}^{3})) 
\end{equation} 
such that 
\begin{equation}
b_{h} \in L^{2} ((0,\infty); H^{m+ \frac{7}{4}} (\mathbb{R}^{3})), b_{v} \in L^{2} ((0,\infty); H^{m+\frac{5}{4}}(\mathbb{R}^{3}))
\end{equation} 
to \eqref{est 78}, and $(u,b) \rvert_{t=0} = (u_{0}, b_{0})$. 
\end{theorem}
We note that Theorem \ref{Theorem 2.5} improves \cite[Corollary 1.4]{Y15} which claimed the global well-posedness of the system \eqref{est 78} when $ \Lambda^{\frac{5}{2}} b_{v}$ in \eqref{est 80} is replaced by $\Lambda^{\frac{7}{2}}b_{v}$. 

\begin{remark}\label{Remark 2.1} 
To the best of our knowledge, Theorems \ref{Theorem 2.1}-\ref{Theorem 2.3} present first component reduction results for regularity criteria of the Hall-MHD system. Moreover, Theorems \ref{Theorem 2.4}-\ref{Theorem 2.5} present first global well-posedness results for the Hall-MHD system with magnetic diffusion that is weaker than the critical threshold for the MHD system added by one more derivative, i.e., $\eta_{h} \Lambda^{3} b_{h} + \eta_{v} \Lambda^{2} b_{v}$ in \eqref{est 64} and $\eta_{h} \Lambda^{\frac{7}{2}} b_{h} + \eta_{v} \Lambda^{\frac{5}{2}} b_{v}$ in \eqref{est 80} are both weaker by one derivative in the vertical component than the hyper-diffusion that we previously thought we need to attain global well-posedness (e.g., \cite[Theorem 2.3]{Y19a} and \cite[Corollary 1.4]{Y15}).  
\begin{enumerate}
\item Concerning Theorems \ref{Theorem 2.1}-\ref{Theorem 2.3}, a natural question may be whether or not one can obtain a regularity criteria in terms of partial derivatives of the pressure $\pi$ (see e.g., \cite{CT08, ZP10} on a criteria in terms of $\partial_{3} \pi$ for the 3-D NS equations, and \cite{CW10} for the 3-D MHD system). Another direction will be to extend such component reduction results to the 3-D Hall-MHD system (see e.g., \cite{Y16a} concerning the regularity criteria in terms of $u_{3}$ and $u_{4}$ for the 4-D NS equations).  
\item Concerning Theorems \ref{Theorem 2.4}-\ref{Theorem 2.5}, reducing the strength of the hyper-diffusion of magnetic field furthermore will certainly be of great interest and significance. 
\item Finally, it would be of interest to extend Theorems \ref{Theorem 2.1}-\ref{Theorem 2.4} to the Hall-MHD system with ion-slip effect (see e.g., \cite{HHM19}). 
\end{enumerate} 
\end{remark}

\section{Proofs of Theorems \ref{Theorem 2.1}-\ref{Theorem 2.4}}\label{Section 4}

\subsection{Proof of Theorem \ref{Theorem 2.1}} 
We take $L^{2}(\mathbb{R}^{2})$-inner products on \eqref{est 1}-\eqref{est 2} with $(u,b)$ and use the fact that 
\begin{equation}\label{est 11}
\int_{\mathbb{R}^{2}} \nabla \times (j\times b) \cdot b dx = \int_{\mathbb{R}^{2}} (j\times b) \cdot j dx \overset{\eqref{key}}{=} 0
\end{equation} 
to deduce the energy identity for all $t \in [0,T]$: 
\begin{equation}\label{est 12} 
\lVert u(t) \rVert_{L^{2}}^{2} + \lVert b(t) \rVert_{L^{2}}^{2} + 2 \int_{0}^{t}\nu \lVert \nabla u\rVert_{L^{2}}^{2} + \eta \lVert \nabla b \rVert_{L^{2}}^{2} ds = \lVert u_{0} \rVert_{L^{2}}^{2} + \lVert b_{0} \rVert_{L^{2}}^{2}.
\end{equation} 
Next, we take $L^{2}(\mathbb{R}^{2})$-inner products on \eqref{est 1}-\eqref{est 2} with $(-\Delta u, -\Delta b)$ to deduce 
\begin{equation}\label{est 13}
\frac{1}{2} \partial_{t} (\lVert \nabla u \rVert_{L^{2}}^{2} + \lVert \nabla b\rVert_{L^{2}}^{2}) + \nu \lVert \Delta u \rVert_{L^{2}}^{2} + \eta \lVert \Delta b \rVert_{L^{2}}^{2} = \sum_{i=1}^{5} \RomanI_{i} 
\end{equation} 
where 
\begin{subequations}\label{est 14}
\begin{align}
\RomanI_{1} \triangleq&  \int_{\mathbb{R}^{2}} (u\cdot\nabla) u \cdot \Delta u dx,\\
\RomanI_{2} \triangleq&  \int_{\mathbb{R}^{2}} (u\cdot\nabla) b \cdot \Delta b dx,\\
\RomanI_{3} \triangleq& -\int_{\mathbb{R}^{2}} (b\cdot\nabla) b \cdot \Delta u dx,\\
\RomanI_{4} \triangleq& -\int_{\mathbb{R}^{2}} (b\cdot\nabla) u \cdot \Delta b dx ,\\
\RomanI_{5} \triangleq& \epsilon \int_{\mathbb{R}^{2}} \nabla \times (j\times b) \cdot \Delta b dx. 
\end{align}
\end{subequations}
The estimates on $\RomanI_{1}-\RomanI_{4}$ are immediate as follows: by integration by parts, H$\ddot{\mathrm{o}}$lder's, Gagliardo-Nirenberg, and Young's inequalities, 
\begin{align}
 \RomanI_{1} + \RomanI_{2} + \RomanI_{3} + \RomanI_{4} 
\overset{\eqref{est 14}}{\lesssim}& \lVert \nabla u \rVert_{L^{2}} ( \lVert \nabla u \rVert_{L^{4}}^{2} + \lVert \nabla b \rVert_{L^{4}}^{2}) 
\lesssim \lVert \nabla u \rVert_{L^{2}} ( \lVert \nabla u \rVert_{L^{2}} \lVert \Delta u \rVert_{L^{2}} + \lVert \nabla b \rVert_{L^{2}} \lVert \Delta b \rVert_{L^{2}}) \nonumber\\
\leq& \frac{\nu}{2} \lVert \Delta u \rVert_{L^{2}}^{2} + \frac{\eta}{4} \lVert \Delta b \rVert_{L^{2}}^{2} + C \lVert \nabla u \rVert_{L^{2}}^{2} ( \lVert \nabla u\rVert_{L^{2}}^{2} + \lVert \nabla b \rVert_{L^{2}}^{2}). \label{est 23} 
\end{align} 
The heart of the matter is, of course, the Hall term. For clarity we first split
\begin{equation}\label{est 16}
\RomanI_{5} = \RomanI_{5,1} + \RomanI_{5,2} \text{ where } \RomanI_{5,1} \triangleq \epsilon \int_{\mathbb{R}^{2}} \nabla \times (j\times b) \cdot \partial_{1}^{2} b dx \text{ and } \RomanI_{5,2} \triangleq \epsilon \int_{\mathbb{R}^{2}} \nabla \times (j\times b) \cdot \partial_{2}^{2} b dx.
\end{equation} 
Now we carefully compute 
\begin{align}
\RomanI_{5,1} \overset{\eqref{est 16}}{=} \epsilon \int_{\mathbb{R}^{2}} (j\times b) \cdot \partial_{1}^{2} j dx=& -\epsilon  \int_{\mathbb{R}^{2}} (\partial_{1} j \times b) \cdot \partial_{1} j + (j \times \partial_{1} b) \cdot \partial_{1} j dx \nonumber\\
\overset{\eqref{key}}{=}& -\epsilon \int_{\mathbb{R}^{2}} (j\times \partial_{1} b) \cdot \partial_{1} j dx  = \sum_{i=1}^{6} \RomanI_{5,1,i}\label{est 21}
\end{align}
where 
\begin{subequations}\label{est 17}
\begin{align}
& \RomanI_{5,1,1} \triangleq - \epsilon \int_{\mathbb{R}^{2}} j_{2}\partial_{1} b_{3} \partial_{1} j_{1} dx, \\
& \RomanI_{5,1,2} \triangleq \epsilon \int_{\mathbb{R}^{2}} j_{3} \partial_{1} b_{2} \partial_{1} j_{1} dx, \\
& \RomanI_{5,1,3} \triangleq \epsilon \int_{\mathbb{R}^{2}} j_{1} \partial_{1} b_{3} \partial_{1} j_{2} dx, \\
& \RomanI_{5,1,4} \triangleq -\epsilon \int_{\mathbb{R}^{2}} j_{3}\partial_{1} b_{1} \partial_{1} j_{2} dx, \\
& \RomanI_{5,1,5} \triangleq -\epsilon \int_{\mathbb{R}^{2}} j_{1}\partial_{1} b_{2} \partial_{1} j_{3} dx, \\
& \RomanI_{5,1,6} \triangleq \epsilon \int_{\mathbb{R}^{2}} j_{2}\partial_{1}b_{1}\partial_{1} j_{3} dx.
\end{align}
\end{subequations} 
We make the key observation that $\RomanI_{5,1,1}$ and $\RomanI_{5,1,3}$ together cancel out as follows:
\begin{align}
\RomanI_{5,1,1} + \RomanI_{5,1,3} \overset{\eqref{est 17}}{=}& -\epsilon  \int_{\mathbb{R}^{2}} j_{2} \partial_{1} b_{3} \partial_{1} j_{1} - j_{1} \partial_{1} b_{3} \partial_{1} j_{2} dx \nonumber\\
\overset{\eqref{j}}{=}& -\epsilon \int_{\mathbb{R}^{2}} - \partial_{1} b_{3} \partial_{1}b_{3}\partial_{1}\partial_{2} b_{3} + \partial_{2} b_{3} \partial_{1} b_{3} \partial_{1}\partial_{1} b_{3}dx \nonumber \\
=& -\epsilon  \int_{\mathbb{R}^{2}} -\partial_{1} b_{3} \frac{1}{2}\partial_{2} (\partial_{1} b_{3})^{2} + \partial_{2} b_{3} \frac{1}{2}\partial_{1} (\partial_{1}b_{3})^{2} dx \nonumber  \\
=& -\epsilon \int_{\mathbb{R}^{2}} \frac{1}{2} \partial_{1}\partial_{2} b_{3} (\partial_{1}b_{3})^{2} - \frac{1}{2} \partial_{1}\partial_{2} b_{3} (\partial_{1} b_{3})^{2} dx = 0. \label{est 19}
\end{align}
Analogously, we compute  
\begin{align}
\RomanI_{5,2} \overset{\eqref{est 16}}{=}\epsilon  \int_{\mathbb{R}^{2}} (j\times b) \cdot \partial_{2}^{2} j dx=& -\epsilon  \int_{\mathbb{R}^{2}} (\partial_{2} j \times b) \cdot \partial_{2} j + (j \times \partial_{2} b) \cdot \partial_{2} j dx \nonumber\\
\overset{\eqref{key}}{=}& -\epsilon \int_{\mathbb{R}^{2}} (j\times \partial_{2} b) \cdot \partial_{2} j dx  = \sum_{i=1}^{6} \RomanI_{5,2,i}\label{est 22} 
\end{align}
where 
\begin{subequations}\label{est 18}
\begin{align}
& \RomanI_{5,2,1} \triangleq - \epsilon \int_{\mathbb{R}^{2}} j_{2}\partial_{2} b_{3} \partial_{2} j_{1} dx, \\
& \RomanI_{5,2,2} \triangleq \epsilon \int_{\mathbb{R}^{2}} j_{3} \partial_{2} b_{2} \partial_{2} j_{1} dx, \\
& \RomanI_{5,2,3} \triangleq \epsilon \int_{\mathbb{R}^{2}} j_{1} \partial_{2} b_{3} \partial_{2} j_{2} dx, \\
& \RomanI_{5,2,4} \triangleq -\epsilon \int_{\mathbb{R}^{2}} j_{3}\partial_{2} b_{1} \partial_{2} j_{2} dx, \\
& \RomanI_{5,2,5} \triangleq -\epsilon \int_{\mathbb{R}^{2}} j_{1}\partial_{2} b_{2} \partial_{2} j_{3} dx, \\
& \RomanI_{5,2,6} \triangleq \epsilon \int_{\mathbb{R}^{2}} j_{2}\partial_{2}b_{1}\partial_{2} j_{3} dx.
\end{align}
\end{subequations} 
Again, we make the key observation that $\RomanI_{5,2,1}$ and $\RomanI_{5,2,3}$ together cancel out as follows:  
\begin{align} 
\RomanI_{5,2,1} + \RomanI_{5,2,3} \overset{\eqref{est 18}}{=}& -\epsilon  \int_{\mathbb{R}^{2}} j_{2} \partial_{2} b_{3} \partial_{2} j_{1} - j_{1} \partial_{2} b_{3} \partial_{2} j_{2} dx \nonumber\\
\overset{\eqref{j}}{=}& -\epsilon \int_{\mathbb{R}^{2}} - \partial_{1} b_{3} \partial_{2}b_{3}\partial_{2}\partial_{2} b_{3} + \partial_{2} b_{3} \partial_{2} b_{3} \partial_{2}\partial_{1} b_{3}dx  \nonumber \\
=& - \epsilon \int_{\mathbb{R}^{2}} -\partial_{1} b_{3} \frac{1}{2}\partial_{2} (\partial_{2} b_{3})^{2} + \partial_{2} b_{3} \frac{1}{2}\partial_{1} (\partial_{2}b_{3})^{2} dx \nonumber  \\
=& -\epsilon \int_{\mathbb{R}^{2}} \frac{1}{2} \partial_{1}\partial_{2} b_{3} (\partial_{2}b_{3})^{2} - \frac{1}{2} \partial_{1}\partial_{2} b_{3} (\partial_{2} b_{3})^{2} dx = 0.\label{est 20}
\end{align}
Thanks to \eqref{est 19} and \eqref{est 20}, we only have to estimate $\RomanI_{5,1,i}$ and $\RomanI_{5,2,i}$ for $i \in \{2,4,5,6\}$ and it turns out that we can immediately obtain from \eqref{est 17} and \eqref{est 18} 
\begin{subequations}\label{est 25}
\begin{align}
& \RomanI_{5,1,2} = \epsilon \int_{\mathbb{R}^{2}} j_{3} \partial_{1} b_{2} \partial_{1} j_{1} dx = -\epsilon \int_{\mathbb{R}^{2}} \partial_{1} (j_{3} \partial_{1} b_{2}) j_{1} dx \lesssim \int_{\mathbb{R}^{2}} \lvert j_{1} \rvert \lvert \nabla b \rvert \lvert \nabla^{2} b \rvert dx, \\
& \RomanI_{5,1,4} = - \epsilon \int_{\mathbb{R}^{2}} j_{3} \partial_{1} b_{1} \partial_{1} j_{2} dx = \epsilon \int_{\mathbb{R}^{2}} \partial_{1} (j_{3} \partial_{1} b_{1}) j_{2} dx \lesssim \int_{\mathbb{R}^{2}} \lvert j_{2} \rvert \lvert \nabla b \rvert \lvert \nabla^{2} b \rvert dx, \\
& \RomanI_{5,1,5} = -\epsilon \int_{\mathbb{R}^{2}} j_{1} \partial_{1} b_{2} \partial_{1} j_{3} dx \lesssim \int_{\mathbb{R}^{2}} \lvert j_{1} \rvert \lvert \nabla b \rvert \lvert \nabla^{2} b \rvert dx, \\
& \RomanI_{5,1,6} = \epsilon \int_{\mathbb{R}^{2}} j_{2} \partial_{1} b_{1} \partial_{1} j_{3} dx \lesssim \int_{\mathbb{R}^{2}} \lvert j_{2} \rvert \lvert \nabla b \rvert \lvert \nabla^{2} b \rvert dx, \\
& \RomanI_{5,2,2} = \epsilon \int_{\mathbb{R}^{2}} j_{3} \partial_{2} b_{2}\partial_{2} j_{1} dx = - \epsilon \int_{\mathbb{R}^{2}} \partial_{2} (j_{3} \partial_{2} b_{2}) j_{1} dx \lesssim \int_{\mathbb{R}^{2}} \lvert j_{1} \rvert \lvert \nabla b \rvert \lvert \nabla^{2} b \rvert dx, \\
& \RomanI_{5,2,4} = -\epsilon  \int_{\mathbb{R}^{2}} j_{3} \partial_{2} b_{1} \partial_{2} j_{2} dx =\epsilon  \int_{\mathbb{R}^{2}} \partial_{2} (j_{3} \partial_{2} b_{1}) j_{2} dx \lesssim \int_{\mathbb{R}^{2}} \lvert j_{2} \rvert \lvert \nabla b \rvert \lvert \nabla^{2} b \rvert dx, \\
& \RomanI_{5,2,5} = -\epsilon \int_{\mathbb{R}^{2}} j_{1} \partial_{2} b_{2}\partial_{2} j_{3} dx \lesssim \int_{\mathbb{R}^{2}} \lvert j_{1} \rvert \lvert \nabla b \rvert \lvert \nabla^{2} b \rvert dx, \\
& \RomanI_{5,2,6} = \epsilon \int_{\mathbb{R}^{2}} j_{2}\partial_{2} b_{1} \partial_{2} j_{3} dx \lesssim \int_{\mathbb{R}^{2}} \lvert j_{2} \rvert \lvert \nabla b \rvert \lvert \nabla^{2} b \rvert dx. 
\end{align}
\end{subequations} 
In fact, the cancellations \eqref{est 19} and \eqref{est 20} are not necessary for this proof of Theorem \ref{Theorem 2.1} because we can bound 
\begin{align*}
\RomanI_{5,1,1} + \RomanI_{5,2,1} \lesssim \int_{\mathbb{R}^{2}} \lvert j_{2} \rvert \lvert \nabla b \rvert \lvert \nabla^{2} b \rvert dx \hspace{1mm} \text{ and } \hspace{1mm} \RomanI_{5,1,3} + \RomanI_{5,2,3} \lesssim \int_{\mathbb{R}^{2}} \lvert j_{1} \rvert \lvert \nabla b \rvert \lvert \nabla^{2} b \rvert dx.
\end{align*}
Nonetheless, they simplified our computations and become absolutely necessary in the proof of Theorems \ref{Theorem 2.2}, \ref{Theorem 2.3} (1), and \ref{Theorem 2.4} (see Remarks \ref{Remark 3.1}, \ref{Remark 3.2}, and \ref{Remark 3.3}). 

Now we are ready to conclude the proof of Theorem \ref{Theorem 2.1}; for $p \in (2,\infty]$, interpreting $\frac{2p}{p-2} = 2$ when $p = \infty$, we compute via H$\ddot{\mathrm{o}}$lder's, Gagliardo-Nirenberg, and Young's inequalities, 
\begin{align}
\RomanI_{5} \overset{\eqref{est 16} \eqref{est 21} \eqref{est 22}}{=} \sum_{i=1}^{6} \RomanI_{5,1,i} + \RomanI_{5,2,i} \overset{\eqref{est 25}}{\lesssim}& \int_{\mathbb{R}^{2}} \lvert \nabla b \rvert \lvert \nabla^{2} b \rvert \sum_{k=1}^{2} \lvert j_{k} \rvert dx \lesssim \lVert \nabla b \rVert_{L^{\frac{2p}{p-2}}} \lVert \Delta b \rVert_{L^{2}} \sum_{k=1}^{2} \lVert j_{k} \rVert_{L^{p}} \nonumber \\ 
 \leq& \frac{\eta}{4} \lVert \Delta b \rVert_{L^{2}}^{2} + C \lVert \nabla b \rVert_{L^{2}}^{2} \sum_{k=1}^{2} \lVert j_{k} \rVert_{L^{p}}^{\frac{2p}{p-2}}.\label{est 24} 
\end{align}
Applying \eqref{est 24} and \eqref{est 23} to \eqref{est 13} and Gronwall's inequality and relying on that $\nabla u, \nabla b \in L_{T}^{2}L_{x}^{2}$ from \eqref{est 12} complete the proof of Theorem \ref{Theorem 2.1}. 

\subsection{Proof of Theorem \ref{Theorem 2.2}}
The proof of Theorem \ref{Theorem 2.2} almost immediately follows from that of Theorem \ref{Theorem 2.1}. We continue to rely on all the computations \eqref{est 12} - \eqref{est 20} so that the proof of Theorem \ref{Theorem 2.1} shows that we only have to bound $\RomanI_{5,1,i}$ and $\RomanI_{5,2,i}$ for $i \in \{2,4,5,6\}$ by a constant multiple of $\int_{\mathbb{R}^{2}} \lvert j_{3} \rvert \lvert \nabla b \rvert \lvert \nabla^{2} b \rvert dx$ this time, in contrast to \eqref{est 25}. We immediately accomplish this task as follows: from \eqref{est 17} and \eqref{est 18} 
\begin{subequations}\label{est 26}
\begin{align}
& \RomanI_{5,1,2} = \epsilon \int_{\mathbb{R}^{2}} j_{3} \partial_{1} b_{2} \partial_{1} j_{1} dx \lesssim \int_{\mathbb{R}^{2}} \lvert j_{3} \rvert \lvert \nabla b \rvert \lvert \nabla^{2} b \rvert dx,   \\
& \RomanI_{5,1,4} = - \epsilon \int_{\mathbb{R}^{2}} j_{3} \partial_{1} b_{1} \partial_{1} j_{2} dx \lesssim \int_{\mathbb{R}^{2}} \lvert j_{3} \rvert \lvert \nabla b \rvert \lvert \nabla^{2} b \rvert dx, \\
& \RomanI_{5,1,5} = -\epsilon \int_{\mathbb{R}^{2}} j_{1} \partial_{1} b_{2} \partial_{1} j_{3} dx = \epsilon \int_{\mathbb{R}^{2}} \partial_{1} (j_{1} \partial_{1} b_{2}) j_{3} dx \lesssim \int_{\mathbb{R}^{2}} \lvert j_{3} \rvert \lvert \nabla b \rvert \lvert \nabla^{2} b \rvert dx,  \\
& \RomanI_{5,1,6} = \epsilon \int_{\mathbb{R}^{2}} j_{2} \partial_{1} b_{1} \partial_{1} j_{3} dx  = - \epsilon  \int_{\mathbb{R}^{2}} \partial_{1} (j_{2} \partial_{1} b_{1}) j_{3} dx\lesssim \int_{\mathbb{R}^{2}} \lvert j_{3} \rvert \lvert \nabla b \rvert \lvert \nabla^{2} b \rvert dx,  \\
& \RomanI_{5,2,2} = \epsilon \int_{\mathbb{R}^{2}} j_{3} \partial_{2} b_{2}\partial_{2} j_{1} dx \lesssim \int_{\mathbb{R}^{2}} \lvert j_{3} \rvert \lvert \nabla b \rvert \lvert \nabla^{2} b \rvert dx,  \\
& \RomanI_{5,2,4} = - \epsilon \int_{\mathbb{R}^{2}} j_{3} \partial_{2} b_{1} \partial_{2} j_{2} dx \lesssim \int_{\mathbb{R}^{2}} \lvert j_{3} \rvert \lvert \nabla b \rvert \lvert \nabla^{2} b \rvert dx,  \\
& \RomanI_{5,2,5} = -\epsilon \int_{\mathbb{R}^{2}} j_{1} \partial_{2} b_{2}\partial_{2} j_{3} dx = \epsilon \int_{\mathbb{R}^{2}} \partial_{2} (j_{1} \partial_{2} b_{2}) j_{3} dx \lesssim \int_{\mathbb{R}^{2}} \lvert j_{3} \rvert \lvert \nabla b \rvert \lvert \nabla^{2} b \rvert dx,\\
& \RomanI_{5,2,6} = \epsilon \int_{\mathbb{R}^{2}} j_{2}\partial_{2} b_{1}  \partial_{2} j_{3} dx = - \epsilon \int_{\mathbb{R}^{2}} \partial_{2} (j_{2} \partial_{2} b_{1}) j_{3} dx \lesssim  \int_{\mathbb{R}^{2}} \lvert j_{3} \rvert \lvert \nabla b \rvert \lvert \nabla^{2} b \rvert dx. 
\end{align}
\end{subequations} 
\begin{remark}\label{Remark 3.1} 
Here, we stress that the cancellations \eqref{est 19} and \eqref{est 20} are crucial because $j_{3}$ is absent in any of $\RomanI_{5,1,1}$ and $\RomanI_{5,1,3}$ in \eqref{est 17} and $\RomanI_{5,2,1}$ and $\RomanI_{5,2,3}$ in \eqref{est 18}. 
\end{remark}
Following the same proof of Theorem \ref{Theorem 2.1}  immediately completes the proof of Theorem \ref{Theorem 2.2}. 

\subsection{Proof of Theorem \ref{Theorem 2.3}} 
First, we note that the energy identity \eqref{est 12} with $\nu = \eta = 1$ remains valid providing us the regularity of 
\begin{equation}\label{est 30}
u, b \in L_{T}^{\infty} L_{x}^{2} \cap L_{T}^{2} \dot{H}_{x}^{1}.  
\end{equation}
The proof will consist of a few steps; for clarity we divide it into three propositions. For the first proposition, let us write down the equation of motion of vorticity $\omega = \nabla \times u$:
\begin{equation}\label{vorticity} 
\partial_{t} \omega + (u\cdot\nabla) \omega = \Delta \omega + (\omega\cdot \nabla) u + \nabla \times ((b\cdot\nabla) b).
\end{equation} 

\begin{proposition}\label{Proposition 3.1}
Under the hypothesis of Theorem \ref{Theorem 2.3}, let $\omega$ be a smooth solution to \eqref{vorticity} over $[0,T]$. Then 
\begin{equation}\label{est 46}
\omega  \in L_{T}^{\infty} L_{x}^{2} \cap L_{T}^{2} \dot{H}_{x}^{1}.
\end{equation} 
\end{proposition}

\begin{proof}[Proof of Proposition \ref{Proposition 3.1}]
We recall the vector calculus identity of 
\begin{align*}
(\nabla \times \Theta) \times \Theta = - \nabla \left( \frac{ \lvert \Theta \rvert^{2}}{2} \right) + (\Theta \cdot \nabla) \Theta \hspace{3mm} \forall \hspace{1mm} \Theta \in \mathbb{R}^{3} 
\end{align*}
so that we may rewrite the Hall term as 
\begin{equation}\label{Hall rewritten}
\nabla \times ( j\times b) = \nabla \times [- \nabla ( \frac{ \lvert b \rvert^{2}}{2} ) + (b\cdot\nabla) b] = \nabla \times ((b\cdot\nabla) b). 
\end{equation} 
Then we define 
\begin{equation}\label{est 28}
z^{1} \triangleq b + \omega; 
\end{equation} 
we note that this trick was used by Chae and Wolf in \cite[equation (2.7)]{CW15}. Then we see from \eqref{est 0}, \eqref{vorticity}, and \eqref{Hall rewritten} that the equation of motion of $z^{1}$ is given by 
\begin{equation}\label{est 29}
\partial_{t} z^{1} + (u\cdot\nabla) z^{1} = (z^{1} \cdot \nabla) u + \Delta z^{1}. 
\end{equation} 
We take $L^{2}(\mathbb{R}^{2})$-inner products with $z^{1}$ in \eqref{est 29} and compute by relying on H$\ddot{\mathrm{o}}$lder's, Gagliardo-Nirenberg, and Young's inequalities 
\begin{align}
\frac{1}{2} \partial_{t} \lVert z^{1} \rVert_{L^{2}}^{2} + \lVert \nabla z^{1} \rVert_{L^{2}}^{2} =& \int_{\mathbb{R}^{2}} (z^{1} \cdot \nabla) u \cdot z^{1} dx \label{est 31}\\
\leq& \lVert z^{1} \rVert_{L^{4}}^{2} \lVert \nabla u \rVert_{L^{2}} 
\lesssim \lVert z^{1} \rVert_{L^{2}} \lVert \nabla z^{1} \rVert_{L^{2}} \lVert \nabla u \rVert_{L^{2}} \leq \frac{1}{2} \lVert \nabla z^{1} \rVert_{L^{2}}^{2}+ C \lVert z^{1} \rVert_{L^{2}}^{2} \lVert \nabla u \rVert_{L^{2}}^{2}. \nonumber 
\end{align} 
Subtracting $\frac{1}{2} \lVert \nabla z^{1} \rVert_{L^{2}}^{2}$ from both sides of \eqref{est 31}, applying Gronwall's inequality, and relying on the fact that $u \in L_{T}^{2} \dot{H}_{x}^{1}$ from \eqref{est 30} give us 
\begin{equation}\label{est 35}
z^{1} \in L_{T}^{\infty} L_{x}^{2} \cap L_{T}^{2} \dot{H}_{x}^{1},
\end{equation} 
and consequently due to \eqref{est 30}
\begin{equation}
\omega \overset{\eqref{est 28}}{=} z^{1} - b \in L_{T}^{\infty} L_{x}^{2} \cap L_{T}^{2} \dot{H}_{x}^{1},
\end{equation} 
which implies \eqref{est 46}.
\end{proof}
 
Next, we define 
\begin{equation}\label{est 32}
z^{2} \triangleq \nabla \times z^{1} \overset{\eqref{est 28}}{=} j + \nabla \times \omega;
\end{equation}  
we see that the governing equation of $z^{2}$ is 
\begin{equation}\label{est 33} 
\partial_{t} z^{2} + (u\cdot\nabla) z^{2} = - (\omega \cdot \nabla) z^{1} + \Delta z^{2} + (z^{1} \cdot \nabla) \omega + (z^{2} \cdot \nabla) u 
+ 2 
\begin{pmatrix}
0\\
0 \\
\partial_{1} z^{1} \cdot \partial_{2} u - \partial_{2} z^{1} \cdot \partial_{1} u
\end{pmatrix}. 
\end{equation} 

\begin{proposition}\label{Proposition 3.2}
Under the hypothesis of Theorem \ref{Theorem 2.3}, let $z^{2}$ be a smooth solution to \eqref{est 33} over $[0,T]$. Then 
\begin{equation}\label{est 36} 
z^{2} \in L_{T}^{\infty} L_{x}^{2} \cap L_{T}^{2} \dot{H}_{x}^{1}. 
\end{equation} 
\end{proposition}

\begin{proof}[Proof of Proposition \ref{Proposition 3.2}]
We take $L^{2}(\mathbb{R}^{2})$-inner products on \eqref{est 33} with $z^{2}$ and estimate via H$\ddot{\mathrm{o}}$lder's, Gagliardo-Nirenberg, and Young's inequalities 
\begin{align}
& \frac{1}{2} \partial_{t} \lVert z^{2} \rVert_{L^{2}}^{2} + \lVert \nabla z^{2} \rVert_{L^{2}}^{2} \nonumber\\
\lesssim& \lVert \omega \rVert_{L^{4}} \lVert \nabla z^{2} \rVert_{L^{2}} \lVert z^{1} \rVert_{L^{4}} + \lVert z^{1} \rVert_{L^{4}} \lVert \nabla z^{2} \rVert_{L^{2}} \lVert \omega \rVert_{L^{4}} + \lVert z^{2} \rVert_{L^{4}}^{2} \lVert \nabla u \rVert_{L^{2}} + \lVert \nabla z^{1} \rVert_{L^{2}} \lVert \nabla u \rVert_{L^{4}} \lVert z^{2} \rVert_{L^{4}} \nonumber \\
\lesssim& \lVert \omega \rVert_{L^{2}}^{\frac{1}{2}} \lVert \nabla \omega \rVert_{L^{2}}^{\frac{1}{2}} \lVert \nabla z^{2} \rVert_{L^{2}} \lVert z^{1} \rVert_{L^{2}}^{\frac{1}{2}} \lVert \nabla z^{1} \rVert_{L^{2}}^{\frac{1}{2}} + \lVert z^{1} \rVert_{L^{2}}^{\frac{1}{2}} \lVert \nabla z^{1} \rVert_{L^{2}}^{\frac{1}{2}} \lVert \nabla z^{2} \rVert_{L^{2}} \lVert \omega \rVert_{L^{2}}^{\frac{1}{2}} \lVert \nabla \omega \rVert_{L^{2}}^{\frac{1}{2}} \nonumber\\
&+ \lVert z^{2} \rVert_{L^{2}} \lVert \nabla z^{2} \rVert_{L^{2}} \lVert \nabla u \rVert_{L^{2}} + \lVert \nabla z^{1} \rVert_{L^{2}} \lVert \omega \rVert_{L^{2}}^{\frac{1}{2}} \lVert \nabla \omega \rVert_{L^{2}}^{\frac{1}{2}} \lVert z^{2} \rVert_{L^{2}}^{\frac{1}{2}} \lVert \nabla z^{2} \rVert_{L^{2}}^{\frac{1}{2}} \nonumber\\
\leq& \frac{1}{2} \lVert \nabla z^{2} \rVert_{L^{2}}^{2}+  C ( \lVert \nabla \omega \rVert_{L^{2}}^{2} + \lVert \nabla z^{1} \rVert_{L^{2}}^{2} + 1) (\lVert \omega \rVert_{L^{2}}^{2} + \lVert z^{1} \rVert_{L^{2}}^{2} + 1) ( \lVert z^{2} \rVert_{L^{2}}^{2} + 1). \label{est 34}
\end{align}
After subtracting $\frac{1}{2} \lVert \nabla z^{2} \rVert_{L^{2}}^{2}$ from both sides of \eqref{est 34}, relying on \eqref{est 46} and \eqref{est 35} and applying Gronwall's inequality allow us to deduce \eqref{est 36}. 
\end{proof}

Now, first we prove Theorem \ref{Theorem 2.3} part (1) with a criteria in terms of $\Delta u_{3}$ in \eqref{est 52}. Considering the definition of $z^{2}$ in \eqref{est 32} and \eqref{est 36}, we realize that the $H^{1}(\mathbb{R}^{2})$-bound of $(u,b)$ is attained once we obtain an $L^{2}(\mathbb{R}^{2})$-bound of $\nabla \times \omega$ that has the governing equation of 
\begin{equation}\label{est 37} 
\partial_{t} \nabla \times \omega + \nabla \times ((u\cdot\nabla) \omega) = \Delta \nabla \times \omega + \nabla \times ((\omega \cdot \nabla) u) + \nabla \times \nabla \times ((b\cdot\nabla) b). 
\end{equation} 
In fact, we consider only the third component of \eqref{est 37}, namely
\begin{equation}\label{est 38}
\partial_{t} (\nabla \times \omega)_{3} + ( \nabla \times ((u\cdot\nabla) \omega))_{3} = \Delta (\nabla \times \omega)_{3} + (\nabla \times ((\omega \cdot \nabla) u))_{3} + ( \nabla \times \nabla \times ((b\cdot\nabla) b))_{3}. 
\end{equation} 

\begin{proposition}\label{Proposition 3.3}
Under the hypothesis of Theorem \ref{Theorem 2.3} (1), let $\nabla \times \omega $ be a smooth solution to \eqref{est 37} over $[0,T]$. Then 
\begin{equation}\label{est 47} 
(\nabla \times \omega)_{3} \in L_{T}^{\infty} L_{x}^{2} \cap L_{T}^{2} \dot{H}_{x}^{1}. 
\end{equation} 
\end{proposition}

\begin{proof}[Proof of Proposition \ref{Proposition 3.3}]
We take $L^{2}(\mathbb{R}^{2})$-inner products on \eqref{est 38}with $(\nabla\times \omega)_{3}$ to deduce 
\begin{equation}\label{est 39}
\frac{1}{2} \partial_{t} \lVert (\nabla \times \omega)_{3} \rVert_{L^{2}}^{2} + \lVert \nabla (\nabla \times \omega)_{3} \rVert_{L^{2}}^{2} = \sum_{i=1}^{3} \RomanII_{i} 
\end{equation} 
where 
\begin{subequations}\label{est 40}
\begin{align}
& \RomanII_{1} \triangleq -\int_{\mathbb{R}^{2}} (\nabla \times ((u\cdot\nabla) \omega))_{3} (\nabla \times \omega)_{3} dx, \\
& \RomanII_{2} \triangleq \int_{\mathbb{R}^{2}} (\nabla \times ((\omega \cdot \nabla) u))_{3} (\nabla \times \omega)_{3} dx, \\
& \RomanII_{3} \triangleq \int_{\mathbb{R}^{2}} (\nabla \times \nabla \times ((b\cdot\nabla) b))_{3} (\nabla \times \omega)_{3} dx. 
\end{align}
\end{subequations} 
For convenience in further computations, we note simple identities: any $f, g$ that are $\mathbb{R}^{3}$-valued and do not depend on $x_{3}$ satisfy 
\begin{equation}\label{est 41}
\int_{\mathbb{R}^{2}} (\nabla \times f)_{1} g_{1} + (\nabla \times f)_{2} g_{2}dx = \int_{\mathbb{R}^{2}} f_{3} (\nabla \times g)_{3} dx,
\end{equation} 
and additionally if $g$ is divergence-free, then they satisfy 
\begin{equation}\label{est 42}
\int_{\mathbb{R}^{2}} (\nabla \times f)_{3} (\nabla \times g)_{3} dx = - \int_{\mathbb{R}^{2}} 
\begin{pmatrix}
f_{1}& f_{2} & 0
\end{pmatrix}^{T} 
\cdot 
\begin{pmatrix}
\Delta g_{1} & \Delta g_{2} & 0 
\end{pmatrix}^{T} dx. 
\end{equation} 
Applying \eqref{est 41}-\eqref{est 42} to \eqref{est 40} leads us to 
\begin{subequations}\label{est 43}
\begin{align}
& \RomanII_{1} = \int_{\mathbb{R}^{2}}  
\begin{pmatrix}
(u\cdot\nabla) \omega_{1} & (u\cdot\nabla) \omega_{2} &0
\end{pmatrix}^{T} 
\cdot 
\begin{pmatrix}
\Delta \omega_{1} & \Delta \omega_{2} & 0 
\end{pmatrix}^{T} 
dx, \\
&\RomanII_{2} = -\int_{\mathbb{R}^{2}}
\begin{pmatrix}
(\omega\cdot \nabla) u_{1} & (\omega\cdot\nabla) u_{2} & 0 
\end{pmatrix}^{T}
\cdot 
\begin{pmatrix}
\Delta \omega_{1} & \Delta \omega_{2} & 0 
\end{pmatrix}^{T}
dx, \\
& \RomanII_{3} = -\int_{\mathbb{R}^{2}} (b\cdot\nabla) b_{3} \Delta (\nabla \times \omega)_{3} dx.
\end{align}
\end{subequations} 
We estimate via H$\ddot{\mathrm{o}}$lder's, Gagliardo-Nirenberg, and Young's inequalities and the fact that $\omega_{1} = \partial_{2} u_{3}, \omega_{2} = - \partial_{1} u_{3}$, 
\begin{align}
\RomanII_{1} \lesssim \sum_{k=1}^{2} \lVert u \rVert_{L^{\infty}} \lVert \nabla \omega_{k} \rVert_{L^{2}} \lVert \Delta \omega_{k} \rVert_{L^{2}} 
\lesssim& \lVert u \rVert_{L^{2}}^{\frac{1}{2}} \lVert \Delta u \rVert_{L^{2}}^{\frac{1}{2}} \lVert \Delta u_{3} \rVert_{L^{2}} \lVert \Delta \nabla u_{3} \rVert_{L^{2}} \nonumber \\
\leq& \frac{1}{4} \lVert \nabla (\nabla \times \omega)_{3} \rVert_{L^{2}}^{2} + C \lVert u \rVert_{L^{2}} \lVert \Delta u \rVert_{L^{2}} \lVert (\nabla \times \omega)_{3} \rVert_{L^{2}}^{2}, \label{est 44}
\end{align} 
and  
\begin{align}
\RomanII_{2} \lesssim& \sum_{k=1}^{2} \lVert \omega \rVert_{L^{4}}  \lVert \nabla u_{k} \rVert_{L^{4}}  \lVert \Delta \omega_{k} \rVert_{L^{2}} \nonumber\\
\lesssim& \lVert \nabla u \rVert_{L^{2}} \lVert \Delta u \rVert_{L^{2}} \lVert \Delta \nabla u_{3} \rVert_{L^{2}} \leq \frac{1}{4} \lVert \nabla (\nabla \times \omega)_{3} \rVert_{L^{2}}^{2} + C  \lVert \nabla u \rVert_{L^{2}}^{2} \lVert \Delta u \rVert_{L^{2}}^{2}. \label{est 45} 
\end{align}
For $\RomanII_{3}$, we first consider $p \in (2,\infty)$ and estimate via H$\ddot{\mathrm{o}}$lder's and Gagliardo-Nirenberg inequalities 
\begin{equation}\label{est 46 again}
\RomanII_{3} \leq \lVert b \rVert_{L^{\frac{2p}{p-2}}} \lVert \nabla b_{3} \rVert_{L^{2}} \lVert \Delta (\nabla \times \omega)_{3} \rVert_{L^{p}} \lesssim\lVert b \rVert_{L^{2}}^{\frac{p-2}{p}} \lVert \nabla b_{3} \rVert_{L^{2}}^{\frac{2+p}{p}} \lVert \Delta (\nabla \times \omega)_{3} \rVert_{L^{p}}. 
\end{equation}
Applying \eqref{est 44}, \eqref{est 45}, and \eqref{est 46 again} to \eqref{est 39} gives us 
\begin{align}
&\partial_{t} \lVert (\nabla \times \omega)_{3} \rVert_{L^{2}}^{2}+  \lVert \nabla (\nabla \times\omega)_{3} \rVert_{L^{2}}^{2} \nonumber \\
\lesssim& \lVert (\nabla \times \omega)_{3} \rVert_{L^{2}}^{2}\lVert u \rVert_{L^{2}} \lVert \Delta u \rVert_{L^{2}} +  \lVert \nabla u \rVert_{L^{2}}^{2} \lVert \Delta u \rVert_{L^{2}}^{2}  + \lVert b \rVert_{L^{2}}^{\frac{p-2}{p}} \lVert \nabla b_{3} \rVert_{L^{2}}^{\frac{2+p}{p}} \lVert \Delta (\nabla \times \omega)_{3} \rVert_{L^{p}}.\label{est 48}
\end{align} 
Integrating \eqref{est 48} over time $[0,t]$ and relying on H$\ddot{\mathrm{o}}$lder's inequality give us 
\begin{align}
& \lVert (\nabla \times \omega)_{3} (t) \rVert_{L^{2}}^{2} + \int_{0}^{t} \lVert \nabla (\nabla \times \omega)_{3} \rVert_{L^{2}}^{2} ds \label{est 49}\\
\overset{\eqref{est 30} \eqref{est 46}}{\lesssim}&  \lVert (\nabla \times \omega)_{3} (0) \rVert_{L^{2}}^{2} +  \int_{0}^{t} \lVert (\nabla \times \omega)_{3}  \rVert_{L^{2}}^{2}  \lVert \Delta u \rVert_{L^{2}} ds + 1 +  \left(\int_{0}^{t} \lVert \Delta (\Delta u_{3} ) \rVert_{L^{p}}^{\frac{2p}{p-2}} ds \right)^{\frac{p-2}{2p}}. \nonumber
\end{align}
To obtain a criteria in terms of $BMO(\mathbb{R}^{2})$-norm in \eqref{est 53}, we denote by $\mathcal{H}^{1}(\mathbb{R}^{2})$ the Hardy space and rely on the duality of $\mathcal{H}^{1}(\mathbb{R}^{2})$ and $BMO(\mathbb{R}^{2})$, as well as the div-curl lemma (e.g. \cite[Theorem 12.1]{L02}) to estimate 
\begin{equation}\label{est 50} 
\RomanII_{3} \overset{\eqref{est 43}}{\lesssim}  \lVert (b\cdot\nabla) b \rVert_{\mathcal{H}^{1}} \lVert \Delta (\nabla \times \omega)_{3} \rVert_{BMO} \lesssim \lVert b \rVert_{L^{2}} \lVert \nabla b \rVert_{L^{2}} \lVert \Delta (\Delta u_{3}) \rVert_{BMO}.  
\end{equation} 
We apply \eqref{est 44}, \eqref{est 45}, and \eqref{est 50} to \eqref{est 39}, integrate the resulting equation over time $[0,t]$ and rely on H$\ddot{\mathrm{o}}$lder's inequality to attain similarly to \eqref{est 49}
\begin{align}
&\lVert (\nabla \times \omega)_{3} (t) \rVert_{L^{2}}^{2} + \int_{0}^{t} \lVert \nabla (\nabla \times \omega)_{3} (s)\rVert_{L^{2}}^{2} ds \label{est 51}\\
\overset{\eqref{est 30} \eqref{est 46}}{\lesssim}&   \lVert (\nabla \times \omega)_{3} (0) \rVert_{L^{2}}^{2} +  \int_{0}^{t} \lVert (\nabla \times \omega)_{3}  \rVert_{L^{2}}^{2} \lVert \Delta u \rVert_{L^{2}} ds + 1 +  \left(\int_{0}^{t} \lVert \Delta (\Delta u_{3} ) \rVert_{BMO}^{2} ds \right)^{\frac{1}{2}}. \nonumber
\end{align}
Due to the hypothesis of \eqref{est 52} and \eqref{est 53}, we see that applying Gronwall's inequality on \eqref{est 49} and \eqref{est 51} implies the desired result of \eqref{est 47} and hence the proof of Proposition \ref{Proposition 3.3} is complete. 
\end{proof}

We are ready to conclude the proof of the first part of Theorem \ref{Theorem 2.3}. By Propositions \ref{Proposition 3.2}-\ref{Proposition 3.3} we realize that 
\begin{equation}\label{est 54}
j_{3} \overset{\eqref{est 32}}{=} z_{3}^{2} - (\nabla \times \omega)_{3} \in L_{T}^{\infty} L_{x}^{2} \cap L_{T}^{2} \dot{H}_{x}^{1}. 
\end{equation} 
We can apply Gagliardo-Nirenberg inequality to deduce from  \eqref{est 54} 
\begin{equation}
\int_{0}^{T} \lVert j_{3} \rVert_{L^{4}}^{4} ds \lesssim \sup_{s\in [0,T]} \lVert j_{3} (s) \rVert_{L^{2}}^{2} \int_{0}^{T} \lVert \nabla j_{3}  \rVert_{L^{2}}^{2} ds \lesssim 1. 
\end{equation} 
Therefore, $j_{3} \in L_{T}^{4}L_{x}^{4}$, allowing us to apply Theorem \ref{Theorem 2.2} to deduce \eqref{est 27} as desired. 
\begin{remark}\label{Remark 3.2}
As we emphasized in Remark \ref{Remark 3.1}, the cancellations \eqref{est 19} and \eqref{est 20} were crucial in the proof of Theorem \ref{Theorem 2.2}. Because the proof of Theorem \ref{Theorem 2.3} (1) relied on Theorem \ref{Theorem 2.2}, the proof of Theorem \ref{Theorem 2.3} (1) in turn also relied crucially on the cancellations \eqref{est 19} and \eqref{est 20}.
\end{remark}

Next, the proof of the second part of Theorem \ref{Theorem 2.3} follows same line of reasonings for the first part. In short, we continue to rely on Propositions \ref{Proposition 3.1}-\ref{Proposition 3.2}, obtain analogous estimates to Proposition \ref{Proposition 3.3} in components $k = 1, 2$ instead of $k = 3$, and thereafter rely on Theorem \ref{Theorem 2.1} instead of Theorem \ref{Theorem 2.2}. Due to similarity, we leave this in the Appendix for completeness. 

\subsection{Proof of Theorem \ref{Theorem 2.4}} 
Local existence of the unique solution to \eqref{est 67} starting from the given $(u_{0}, b_{0})$ can be proven following previous works such as \cite[Theorem 2.2]{CDL14}. Moreover, the proof of the blow-up criterion from \eqref{est 4} can be easily seen to go through for \eqref{est 67} because \eqref{est 67} has more diffusive strength than the classical Hall-MHD system. Therefore, it suffices to obtain $H^{1}(\mathbb{R}^{2})$-bound again because it will imply $\int_{0}^{T} \lVert \Delta b \rVert_{L^{2}}^{2} dt < \infty$ from the magnetic diffusion which can bound $\int_{0}^{T} \lVert j \rVert_{BMO}^{2} dt < \infty$ (e.g., \cite[Theorem 1.48]{BCD11}). Hence, we are able to follow the same line of reasoning in the proof of Theorem \ref{Theorem 2.1}. In fact, as we will emphasize in Remark \ref{Remark 3.3}, the key to the proof of Theorem \ref{Theorem 2.4} is once again the cancellations in \eqref{est 19} and \eqref{est 20}. 

First, relying on \eqref{est 11} we obtain the following energy identity:
\begin{equation}\label{est 69} 
\lVert u(t) \rVert_{L^{2}}^{2} + \lVert b(t) \rVert_{L^{2}}^{2} + 2 \int_{0}^{t} \nu \lVert \nabla u \rVert_{L^{2}}^{2} + \eta_{h} \lVert \Lambda^{\frac{3}{2}} b_{h} \rVert_{L^{2}}^{2} + \eta_{v} \lVert \nabla b_{v} \rVert_{L^{2}}^{2} ds = \lVert u_{0} \rVert_{L^{2}}^{2} + \lVert b_{0} \rVert_{L^{2}}^{2}. 
\end{equation} 
We take $L^{2}(\mathbb{R}^{2})$-inner products on \eqref{est 67} with $(-\Delta u, -\Delta b)$ to deduce 
\begin{equation}\label{est 70} 
\frac{1}{2} \partial_{t} (\lVert \nabla u \rVert_{L^{2}}^{2} + \lVert \nabla b\rVert_{L^{2}}^{2}) + \nu \lVert \Delta u \rVert_{L^{2}}^{2} + \eta_{h} \lVert \Lambda^{\frac{5}{2}} b_{h} \rVert_{L^{2}}^{2} + \eta_{v} \lVert \Delta b_{v} \rVert_{L^{2}}^{2} = \sum_{i=1}^{5} \RomanI_{i} 
\end{equation} 
where $\RomanI_{i}$ for $i \in \{1,\hdots, 5\}$ are same as those in \eqref{est 14}, which we recall here for convenience:
\begin{align*}
\RomanI_{1} \triangleq&  \int_{\mathbb{R}^{2}} (u\cdot\nabla) u \cdot \Delta u dx,\\
\RomanI_{2} \triangleq&  \int_{\mathbb{R}^{2}} (u\cdot\nabla) b \cdot \Delta b dx,\\
\RomanI_{3} \triangleq& -\int_{\mathbb{R}^{2}} (b\cdot\nabla) b \cdot \Delta u dx,\\
\RomanI_{4} \triangleq& -\int_{\mathbb{R}^{2}} (b\cdot\nabla) u \cdot \Delta b dx ,\\
\RomanI_{5} \triangleq& \epsilon \int_{\mathbb{R}^{2}} \nabla \times (j\times b) \cdot \Delta b dx. 
\end{align*}
Similarly to \eqref{est 23} we can estimate by H$\ddot{\mathrm{o}}$lder's, Gagliardo-Nirenberg, and Young's inequality, as well as Sobolev embedding $\dot{H}^{\frac{1}{2}}(\mathbb{R}^{2}) \hookrightarrow L^{4}(\mathbb{R}^{2})$, 
\begin{align}
&\RomanI_{1} + \RomanI_{2} + \RomanI_{3} + \RomanI_{4} \overset{\eqref{est 23}}{\lesssim} \lVert \nabla u \rVert_{L^{2}} \lVert \nabla u \rVert_{L^{4}}^{2} + \lVert \nabla u \rVert_{L^{2}} \lVert \nabla b_{h} \rVert_{L^{4}}^{2} + \lVert \nabla u \rVert_{L^{2}} \lVert \nabla b_{v} \rVert_{L^{4}}^{2} \nonumber \\
&\hspace{15mm} \lesssim \lVert \nabla u \rVert_{L^{2}}^{2} \lVert \Delta u \rVert_{L^{2}} + \lVert \nabla u \rVert_{L^{2}} \lVert \Lambda^{\frac{3}{2}} b_{h} \rVert_{L^{2}}^{2} + \lVert \nabla u \rVert_{L^{2}} \lVert \nabla b_{v} \rVert_{L^{2}} \lVert \Delta b_{v} \rVert_{L^{2}} \nonumber\\
&\hspace{15mm} \leq \frac{\nu}{2} \lVert \Delta u \rVert_{L^{2}}^{2} + \frac{\eta_{v}}{2} \lVert \Delta b_{v} \rVert_{L^{2}}^{2} + C (1+ \lVert \nabla u \rVert_{L^{2}}^{2} + \lVert \nabla b \rVert_{L^{2}}^{2})(\lVert \nabla u \rVert_{L^{2}}^{2} + \lVert \Lambda^{\frac{3}{2}} b_{h} \rVert_{L^{2}}^{2}).\label{est 71} 
\end{align}
For the Hall term $\RomanI_{5}$ in \eqref{est 14}, we split again 
\begin{equation}\label{est 72} 
\RomanI_{5} \overset{\eqref{est 16}}{=} \RomanI_{5,1} + \RomanI_{5,2} \overset{\eqref{est 21} \eqref{est 22}}{=} \sum_{i=1}^{6} \RomanI_{5,1,i} + \RomanI_{5,2,i} 
\end{equation} 
where $\RomanI_{5,1,1} + \RomanI_{5,1,3} = \RomanI_{5,2,1} + \RomanI_{5,2,3} = 0$ due to \eqref{est 19} and \eqref{est 20} so that we only need to estimate $\RomanI_{5,1,i}$ and $\RomanI_{5,2,i}$ for $i \in \{2,4,5,6\}$ from \eqref{est 17} and \eqref{est 18}. It turns out that we may bound them all as follows:
\begin{subequations}\label{est 73} 
\begin{align}
 \RomanI_{5,1,2} =& \epsilon  \int_{\mathbb{R}^{2}} j_{3} \partial_{1} b_{2} \partial_{1} j_{1} dx =\epsilon    \int_{\mathbb{R}^{2}} (\partial_{1} b_{2} - \partial_{2} b_{1}) \partial_{1} b_{2} \partial_{1}\partial_{2} b_{3} dx \nonumber\\
&= - \epsilon   \int_{\mathbb{R}^{2}} \partial_{1} [ (\partial_{1} b_{2} -\partial_{2} b_{1}) \partial_{1} b_{2} ] \partial_{2} b_{3} dx \lesssim \int_{\mathbb{R}^{2}} \lvert \nabla b_{h} \rvert \lvert \nabla^{2} b_{h} \rvert \lvert \nabla b \rvert dx,\\
 \RomanI_{5,1,4} =& -\epsilon  \int_{\mathbb{R}^{2}} j_{3} \partial_{1} b_{1} \partial_{1} j_{2} dx =\epsilon   \int_{\mathbb{R}^{2}} (\partial_{1} b_{2} - \partial_{2} b_{1}) \partial_{1} b_{1} \partial_{1}\partial_{1} b_{3} dx \nonumber\\
 &= - \epsilon  \int_{\mathbb{R}^{2}} \partial_{1} [(\partial_{1} b_{2} - \partial_{2} b_{1}) \partial_{1} b_{1}] \partial_{1} b_{3} dx \lesssim \int_{\mathbb{R}^{2}} \lvert \nabla b_{h} \rvert \lvert \nabla^{2} b_{h} \rvert \lvert \nabla b \rvert dx,\\
 \RomanI_{5,1,5} =& -\epsilon  \int_{\mathbb{R}^{2}} j_{1} \partial_{1} b_{2}\partial_{1} j_{3} dx \nonumber\\
 &= -\epsilon  \int_{\mathbb{R}^{2}} \partial_{2} b_{3} \partial_{1} b_{2} \partial_{1} (\partial_{1} b_{2} - \partial_{2} b_{1}) dx \lesssim \int_{\mathbb{R}^{2}} \lvert \nabla b \rvert \lvert \nabla b_{h} \rvert \lvert \nabla^{2} b_{h} \rvert dx, \\
 \RomanI_{5,1,6} =& \epsilon  \int_{\mathbb{R}^{2}} j_{2} \partial_{1} b_{1} \partial_{1} j_{3} dx \nonumber\\
 &= -\epsilon  \int_{\mathbb{R}^{2}} \partial_{1} b_{3} \partial_{1} b_{1} \partial_{1} (\partial_{1} b_{2} - \partial_{2} b_{1}) dx \lesssim \int_{\mathbb{R}^{2}} \lvert \nabla b \rvert \lvert \nabla b_{h} \rvert \lvert \nabla^{2} b_{h} \rvert dx,  \\
 \RomanI_{5,2,2} =& \epsilon  \int_{\mathbb{R}^{2}} j_{3} \partial_{2} b_{2} \partial_{2} j_{1} dx = \epsilon  \int_{\mathbb{R}^{2}} (\partial_{1} b_{2} - \partial_{2} b_{1}) \partial_{2} b_{2}\partial_{2}\partial_{2} b_{3} dx \nonumber\\
 & = -\epsilon   \int_{\mathbb{R}^{2}}\partial_{2} [(\partial_{1} b_{2} - \partial_{2} b_{1}) \partial_{2} b_{2} ] \partial_{2} b_{3} dx \lesssim \int_{\mathbb{R}^{2}} \lvert \nabla b_{h} \rvert \lvert \nabla^{2} b_{h} \rvert \lvert \nabla b \rvert dx, \\
 \RomanI_{5,2,4} =&-\epsilon  \int_{\mathbb{R}^{2}} j_{3} \partial_{2} b_{1} \partial_{2} j_{2}dx = \epsilon  \int_{\mathbb{R}^{2}} (\partial_{1} b_{2} - \partial_{2} b_{1}) \partial_{2} b_{1}\partial_{2}\partial_{1} b_{3} dx \nonumber\\
 & = -\epsilon  \int_{\mathbb{R}^{2}} \partial_{2} [(\partial_{1} b_{2} - \partial_{2} b_{1}) \partial_{2} b_{1} ] \partial_{1} b_{3} dx \lesssim \int_{\mathbb{R}^{2}} \lvert \nabla b_{h} \rvert \lvert \nabla^{2} b_{h} \rvert \lvert \nabla b \rvert dx,  \\
 \RomanI_{5,2,5} =& -\epsilon   \int_{\mathbb{R}^{2}} j_{1} \partial_{2} b_{2}\partial_{2} j_{3} dx \nonumber\\
 &=-\epsilon  \int_{\mathbb{R}^{2}}\partial_{2} b_{3}\partial_{2}b_{2}\partial_{2}(\partial_{1} b_{2} - \partial_{2} b_{1}) dx \lesssim \int_{\mathbb{R}^{2}} \lvert \nabla b \rvert \lvert \nabla b_{h} \rvert \lvert \nabla^{2} b_{h} \rvert dx, \\
 \RomanI_{5,2,6} =&\epsilon   \int_{\mathbb{R}^{2}} j_{2} \partial_{2} b_{1}\partial_{2} j_{3} dx \nonumber\\
 & = -\epsilon  \int_{\mathbb{R}^{2}} \partial_{1} b_{3} \partial_{2} b_{1}\partial_{2} (\partial_{1} b_{2} - \partial_{2} b_{1}) dx \lesssim \int_{\mathbb{R}^{2}} \lvert \nabla b \rvert \lvert \nabla b_{h} \rvert \lvert \nabla^{2} b_{h} \rvert dx. 
\end{align}
\end{subequations} 
\begin{remark}\label{Remark 3.3}
Here, we emphasize again the importance of the cancellations in \eqref{est 19} and \eqref{est 20}; e.g., it is not clear how we can bound 
\begin{align*}
\RomanI_{5,1,1} \overset{\eqref{est 17}}{=} -\epsilon \int_{\mathbb{R}^{2}} j_{2} \partial_{1} b_{3} \partial_{1} j_{1} dx = \epsilon \int_{\mathbb{R}^{2}} (\partial_{1} b_{3})^{2} \partial_{1} \partial_{2} b_{3} dx 
\end{align*}
by a constant multiple of $\int_{\mathbb{R}^{2}} \lvert \nabla b \rvert \lvert \nabla b_{h} \rvert \lvert \nabla^{2} b_{h} \rvert dx$. 
\end{remark}
Due to \eqref{est 73}, we are able to bound by H$\ddot{\mathrm{o}}$lder's inequality, Sobolev embedding $\dot{H}^{\frac{1}{2}}(\mathbb{R}^{2}) \hookrightarrow L^{4}(\mathbb{R}^{2})$, and Young's inequality 
\begin{align}
\RomanI_{5} \overset{\eqref{est 72}\eqref{est 73}}{\lesssim}& \lVert \nabla b_{h} \rVert_{L^{4}} \lVert \nabla^{2} b_{h} \rVert_{L^{4}} \lVert \nabla b \rVert_{L^{2}} \nonumber\\
\lesssim& \lVert \Lambda^{\frac{3}{2}} b_{h} \rVert_{L^{2}} \lVert \Lambda^{\frac{5}{2}} b_{h} \rVert_{L^{2}} \lVert \nabla b \rVert_{L^{2}} \leq \frac{\eta_{h}}{2} \lVert \Lambda^{\frac{5}{2}} b_{h} \rVert_{L^{2}}^{2} + C\lVert \Lambda^{\frac{3}{2}} b_{h} \rVert_{L^{2}}^{2} \lVert \nabla b \rVert_{L^{2}}^{2}.  \label{est 74}
\end{align} 
Applying \eqref{est 71} and \eqref{est 74} to \eqref{est 70} gives us 
\begin{align}
&\partial_{t} (\lVert \nabla u \rVert_{L^{2}}^{2} + \lVert \nabla b \rVert_{L^{2}}^{2})+ \nu \lVert \Delta u \rVert_{L^{2}}^{2} + \eta_{h} \lVert \Lambda^{\frac{5}{2}} b_{h} \rVert_{L^{2}}^{2} + \eta_{v} \lVert \Delta b_{v} \rVert_{L^{2}}^{2} \nonumber\\
& \hspace{15mm} \lesssim (1+ \lVert \nabla u \rVert_{L^{2}}^{2} + \lVert \nabla b \rVert_{L^{2}}^{2}) (\lVert \nabla u \rVert_{L^{2}}^{2} + \lVert \Lambda^{\frac{3}{2}} b_{h} \rVert_{L^{2}}^{2}).\label{est 75}
\end{align} 
Applying Gronwall's inequality on \eqref{est 75} and relying on \eqref{est 69} complete the proof of the $H^{1}(\mathbb{R}^{2})$-bound and therefore that of Theorem \ref{Theorem 2.4}. 

\subsection{Proof of Theorem \ref{Theorem 2.5}}  
Local existence of the unique solution to \eqref{est 78} starting from the given $(u_{0}, b_{0})$ can be proven following previous works such as \cite[Theorem 2.2]{CDL14}. Relying on \eqref{est 11} we obtain the following energy identity:
\begin{equation}\label{est 81} 
\lVert u(t) \rVert_{L^{2}}^{2} + \lVert b(t) \rVert_{L^{2}}^{2} + 2 \int_{0}^{t} \nu \lVert \Lambda^{\frac{5}{4}} u \rVert_{L^{2}}^{2} + \eta_{h} \lVert \Lambda^{\frac{7}{4}} b_{h} \rVert_{L^{2}}^{2} + \eta_{v} \lVert \Lambda^{\frac{5}{4}} b_{v} \rVert_{L^{2}}^{2} ds = \lVert u_{0} \rVert_{L^{2}}^{2} + \lVert b_{0} \rVert_{L^{2}}^{2}. 
\end{equation} 
Although an $H^{1}(\mathbb{R}^{3})$-bound does not suffice for the solution to the 3-D Hall-MHD system to bootstrap to higher regularity, it does for \eqref{est 78} due to its hyper-diffusion; that is the content of the following proposition, the proof of which is left in the Appendix for completeness. 
\begin{proposition}\label{Proposition 3.4} 
Under the hypothesis of Theorem \ref{Theorem 2.5}, if $(u,b)$ is a solution to \eqref{est 78} starting from the given $(u_{0}, b_{0})$ over $[0,T]$, then for all $t \in (0, T]$, $(u,b)$ satisfies 
\begin{align}
& \lVert u(t) \rVert_{H^{m}}^{2} + \lVert b(t) \rVert_{H^{m}}^{2} + \int_{0}^{t} \nu \lVert \Lambda^{\frac{5}{4}} u \rVert_{H^{m}}^{2} + \eta_{h} \lVert \Lambda^{\frac{7}{4}} b_{h} \rVert_{H^{m}}^{2} + \eta_{v}\lVert \Lambda^{\frac{5}{4}} b_{v} \rVert_{H^{m}}^{2} ds \nonumber \\
\lesssim& (1+ \lVert u_{0} \rVert_{H^{m}}^{2} + \lVert b_{0} \rVert_{H^{m}}^{2}) e^{\int_{0}^{t} 1+ \lVert \Lambda^{\frac{5}{4}} u \rVert_{L^{2}}^{2} + \lVert \Lambda^{\frac{9}{4}} b \rVert_{L^{2}}^{2} ds}.  \label{est 85}
\end{align}
\end{proposition} 
Due to Proposition \ref{Proposition 3.4} and \eqref{est 81}, the proof of Theorem \ref{Theorem 2.5} boils down to proving that $\int_{0}^{T} \lVert \Lambda^{\frac{9}{4}} b \rVert_{L^{2}}^{2} ds < \infty$. In this endeavor, we take $L^{2}(\mathbb{R}^{3})$-inner products on \eqref{est 78} with $(-\Delta u, -\Delta b)$ to deduce 
\begin{equation}\label{est 82} 
\frac{1}{2} \partial_{t} (\lVert \nabla u \rVert_{L^{2}}^{2} + \lVert \nabla b\rVert_{L^{2}}^{2}) + \nu \lVert \Lambda^{\frac{9}{4}} u \rVert_{L^{2}}^{2} + \eta_{h} \lVert \Lambda^{\frac{11}{4}} b_{h} \rVert_{L^{2}}^{2} + \eta_{v} \lVert \Lambda^{\frac{9}{4}} b_{v} \rVert_{L^{2}}^{2} = \sum_{i=1}^{5} \RomanIV_{i} 
\end{equation} 
where 
\begin{align*}
\RomanIV_{1} \triangleq&  \int_{\mathbb{R}^{3}} (u\cdot\nabla) u \cdot \Delta u dx,\\
\RomanIV_{2} \triangleq&  \int_{\mathbb{R}^{3}} (u\cdot\nabla) b \cdot \Delta b dx,\\
\RomanIV_{3} \triangleq& -\int_{\mathbb{R}^{3}} (b\cdot\nabla) b \cdot \Delta u dx,\\
\RomanIV_{4} \triangleq& -\int_{\mathbb{R}^{3}} (b\cdot\nabla) u \cdot \Delta b dx ,\\
\RomanIV_{5} \triangleq& \epsilon \int_{\mathbb{R}^{3}} \nabla \times (j\times b) \cdot \Delta b dx. 
\end{align*}
Using H$\ddot{\mathrm{o}}$lder's inequality, the Sobolev embedding $\dot{H}^{\frac{3}{4}}(\mathbb{R}^{3})\hookrightarrow L^{4}(\mathbb{R}^{3})$, Gagliardo-Nirenberg and Young's inequalities, we can compute similarly to \eqref{est 71}  
\begin{align}
&\RomanIV_{1} + \RomanIV_{2} + \RomanIV_{3} + \RomanIV_{4} \lesssim \lVert \nabla u \rVert_{L^{2}} \lVert \nabla u \rVert_{L^{4}}^{2} + \lVert \nabla u \rVert_{L^{2}} \lVert \nabla b_{h} \rVert_{L^{4}}^{2} + \lVert \nabla u \rVert_{L^{2}} \lVert \nabla b_{v} \rVert_{L^{4}}^{2} \label{est 86}\\
\lesssim& \lVert \nabla u \rVert_{L^{2}} \lVert \Lambda^{\frac{5}{4}} u \rVert_{L^{2}} \lVert \Lambda^{\frac{9}{4}} u \rVert_{L^{2}} + \lVert \nabla u \rVert_{L^{2}} \lVert \Lambda^{\frac{7}{4}} b_{h} \rVert_{L^{2}}^{2} + \lVert \nabla u \rVert_{L^{2}} \lVert \Lambda^{\frac{5}{4}} b_{v} \rVert_{L^{2}} \lVert \Lambda^{\frac{9}{4}} b_{v} \rVert_{L^{2}} \nonumber \\
\leq& \frac{\nu}{2} \lVert \Lambda^{\frac{9}{4}} u \rVert_{L^{2}}^{2} + \frac{\eta_{v}}{2} \lVert \Lambda^{\frac{9}{4}} b_{v} \rVert_{L^{2}}^{2} + C (1+ \lVert \nabla u \rVert_{L^{2}}^{2} + \lVert \nabla b \rVert_{L^{2}}^{2})(\lVert \Lambda^{\frac{5}{4}} u \rVert_{L^{2}}^{2} + \lVert \Lambda^{\frac{7}{4}} b_{h} \rVert_{L^{2}}^{2} + \lVert \Lambda^{\frac{5}{4}} b_{v} \rVert_{L^{2}}^{2}). \nonumber 
\end{align}
Now we decompose the Hall term carefully as follows: 
\begin{equation}\label{est 87} 
\RomanIV_{5} = \sum_{k=1}^{3} \epsilon \int_{\mathbb{R}^{3}} \nabla \times (j\times b) \cdot \partial_{k}^{2} b dx  
\end{equation} 
where for any $k \in \{1,2,3\}$, 
\begin{align}
\int_{\mathbb{R}^{3}} \nabla \times (j\times b) \cdot \partial_{k}^{2} b dx 
=&  - \int_{\mathbb{R}^{3}} \partial_{k} (j\times b) \cdot \partial_{k} j dx \nonumber\\ 
\overset{\eqref{key}}{=}&  -\int_{\mathbb{R}^{3}}   j\times \partial_{k} b \cdot \partial_{k} j dx = \sum_{i=1}^{6} \RomanV_{k,i} \label{est 88} 
\end{align}
and 
\begin{subequations}\label{est 89} 
\begin{align}
\RomanV_{k,1} \triangleq& -\int_{\mathbb{R}^{3}} j_{2}\partial_{k} b_{3}\partial_{k} j_{1} dx, \\
\RomanV_{k,2} \triangleq& \int_{\mathbb{R}^{3}} j_{3} \partial_{k} b_{2} \partial_{k} j_{1} dx, \\
\RomanV_{k,3} \triangleq& \int_{\mathbb{R}^{3}} j_{1}\partial_{k} b_{3} \partial_{k} j_{2} dx, \\
\RomanV_{k,4} \triangleq& -\int_{\mathbb{R}^{3}} j_{3}\partial_{k}b_{1}\partial_{k} j_{2} dx, \\
\RomanV_{k,5} \triangleq& -\int_{\mathbb{R}^{3}} j_{1}\partial_{k} b_{2} \partial_{k} j_{3} dx, \\
\RomanV_{k,6} \triangleq& \int_{\mathbb{R}^{3}} j_{2} \partial_{k}b_{1}\partial_{k} j_{3} dx. 
\end{align}
\end{subequations} 
We take the lessons we learned from \eqref{est 19} and \eqref{est 20} and first couple strategically from \eqref{est 89} 
\begin{align}
\RomanV_{k,1} + \RomanV_{k,3} =& -\int_{\mathbb{R}^{3}} j_{2} \partial_{k} b_{3} \partial_{k} j_{1} - j_{1} \partial_{k} b_{3} \partial_{k} j_{2} dx \nonumber \\
=& -\int_{\mathbb{R}^{3}}\partial_{k} b_{3} (j_{2}\partial_{k} j_{1} - j_{1}\partial_{k} j_{2}) dx = \sum_{l=1}^{8} \RomanVI_{k,l} \label{est 90}
\end{align}
where 
\begin{subequations}\label{est 91}
\begin{align}
\RomanVI_{k,1} \triangleq& \int_{\mathbb{R}^{3}}\partial_{k} b_{3} \partial_{1} b_{3}\partial_{k}\partial_{2} b_{3} dx, \\
\RomanVI_{k,2} \triangleq& -\int_{\mathbb{R}^{3}}\partial_{k}b_{3}\partial_{1}b_{3}\partial_{k}\partial_{3} b_{2} dx, \\
\RomanVI_{k,3} \triangleq& -\int_{\mathbb{R}^{3}} \partial_{k} b_{3}\partial_{3} b_{1} \partial_{k}\partial_{2} b_{3} dx, \\
\RomanVI_{k,4} \triangleq& \int_{\mathbb{R}^{3}} \partial_{k} b_{3}\partial_{3} b_{1} \partial_{k} \partial_{3} b_{2} dx, \\
\RomanVI_{k,5} \triangleq& - \int_{\mathbb{R}^{3}}\partial_{k} b_{3} \partial_{2} b_{3}\partial_{k}\partial_{1} b_{3} dx, \\
\RomanVI_{k,6} \triangleq& \int_{\mathbb{R}^{3}} \partial_{k} b_{3}\partial_{2}b_{3}\partial_{k}\partial_{3} b_{1} dx, \\
\RomanVI_{k,7} \triangleq&\int_{\mathbb{R}^{3}} \partial_{k} b_{3} \partial_{3} b_{2}\partial_{k}\partial_{1} b_{3} dx, \\
\RomanVI_{k,8} \triangleq&- \int_{\mathbb{R}^{3}} \partial_{k} b_{3}\partial_{3} b_{2}\partial_{k}\partial_{3} b_{1} dx. 
\end{align}
\end{subequations} 
Then we observe that analogously to \eqref{est 19} and \eqref{est 20}
\begin{align}
\RomanVI_{k,1} + \RomanVI_{k,5} =& \int_{\mathbb{R}^{3}} \partial_{k} b_{3} \partial_{1} b_{3} \partial_{k} \partial_{2} b_{3} - \partial_{k} b_{3}\partial_{2} b_{3} \partial_{k} \partial_{1} b_{3} dx \nonumber \\
=& \int_{\mathbb{R}^{3}} \partial_{1} b_{3} \frac{1}{2} \partial_{2} (\partial_{k} b_{3})^{2} -\partial_{2} b_{3} \frac{1}{2}\partial_{1} (\partial_{k} b_{3})^{2} dx\nonumber \\
=& -\frac{1}{2} \int_{\mathbb{R}^{3}} \partial_{1}\partial_{2} b_{3}(\partial_{k} b_{3})^{2} - \partial_{1}\partial_{2} b_{3} (\partial_{k} b_{3})^{2} dx = 0. \label{new1}
\end{align}
Remarkably, we are able to find two more cancellations upon coupling within \eqref{est 89} appropriately. Let us couple 
\begin{align}
\RomanV_{k,2} + \RomanV_{k,5} =& \int_{\mathbb{R}^{3}} j_{3} \partial_{k} b_{2} \partial_{k} j_{1} - j_{1} \partial_{k} b_{2} \partial_{k} j_{3} dx \nonumber \\
=& \int_{\mathbb{R}^{3}}\partial_{k} b_{2} (j_{3}\partial_{k} j_{1} - j_{1}\partial_{k} j_{3}) dx = \sum_{l=1}^{8} \RomanVII_{k,l} \label{est 92}
\end{align}
where 
\begin{subequations}\label{est 93} 
\begin{align}
\RomanVII_{k,1} \triangleq& \int_{\mathbb{R}^{3}}\partial_{k} b_{2} \partial_{1} b_{2}\partial_{k}\partial_{2} b_{3} dx, \\
\RomanVII_{k,2} \triangleq& -\int_{\mathbb{R}^{3}}\partial_{k}b_{2}\partial_{1}b_{2}\partial_{k}\partial_{3} b_{2} dx, \\
\RomanVII_{k,3} \triangleq& -\int_{\mathbb{R}^{3}} \partial_{k} b_{2}\partial_{2} b_{1} \partial_{k}\partial_{2} b_{3} dx, \\
\RomanVII_{k,4} \triangleq& \int_{\mathbb{R}^{3}} \partial_{k} b_{2}\partial_{2} b_{1} \partial_{k} \partial_{3} b_{2} dx, \\
\RomanVII_{k,5} \triangleq& - \int_{\mathbb{R}^{3}}\partial_{k} b_{2} \partial_{2} b_{3}\partial_{k}\partial_{1} b_{2} dx, \\
\RomanVII_{k,6} \triangleq& \int_{\mathbb{R}^{3}} \partial_{k} b_{2}\partial_{2}b_{3}\partial_{k}\partial_{2} b_{1} dx, \\
\RomanVII_{k,7} \triangleq&\int_{\mathbb{R}^{3}} \partial_{k} b_{2} \partial_{3} b_{2}\partial_{k}\partial_{1} b_{2} dx, \\
\RomanVII_{k,8} \triangleq& -\int_{\mathbb{R}^{3}} \partial_{k} b_{2}\partial_{3} b_{2}\partial_{k}\partial_{2} b_{1} dx. 
\end{align}
\end{subequations} 
We observe that $\RomanVII_{k,2} + \RomanVII_{k,7}$ vanishes together as follows: 
\begin{align}
\RomanVII_{k,2} + \RomanVII_{k,7} =& \int_{\mathbb{R}^{3}} -\partial_{k} b_{2} \partial_{1} b_{2} \partial_{k} \partial_{3} b_{2} +  \partial_{k} b_{2}\partial_{3} b_{2} \partial_{k} \partial_{1} b_{2} dx \nonumber \\
=& \int_{\mathbb{R}^{3}} -\partial_{1} b_{2} \frac{1}{2} \partial_{3} (\partial_{k} b_{2})^{2} + \partial_{3} b_{2} \frac{1}{2}\partial_{1} (\partial_{k} b_{2})^{2} dx\nonumber \\
=& \frac{1}{2} \int_{\mathbb{R}^{3}} \partial_{1}\partial_{3} b_{2}(\partial_{k} b_{2})^{2} - \partial_{1}\partial_{3} b_{2} (\partial_{k} b_{3})^{2} dx = 0. \label{new2}
\end{align}
Third, we couple from \eqref{est 89} 
\begin{align}
\RomanV_{k,4} + \RomanV_{k,6} =& -\int_{\mathbb{R}^{3}} j_{3} \partial_{k} b_{1} \partial_{k} j_{2} - j_{2} \partial_{k} b_{1} \partial_{k} j_{3} dx \nonumber \\
=& -\int_{\mathbb{R}^{3}}\partial_{k} b_{1} (j_{3}\partial_{k} j_{2} - j_{2}\partial_{k} j_{3}) dx = \sum_{l=1}^{8} \RomanVIII_{k,l} \label{est 94}
\end{align}
where 
\begin{subequations}\label{est 95} 
\begin{align}
\RomanVIII_{k,1} \triangleq& \int_{\mathbb{R}^{3}}\partial_{k} b_{1} \partial_{1} b_{2}\partial_{k}\partial_{1} b_{3} dx, \\
\RomanVIII_{k,2} \triangleq& -\int_{\mathbb{R}^{3}}\partial_{k}b_{1}\partial_{1}b_{2}\partial_{k}\partial_{3} b_{1} dx, \\
\RomanVIII_{k,3} \triangleq& -\int_{\mathbb{R}^{3}} \partial_{k} b_{1}\partial_{2} b_{1} \partial_{k}\partial_{1} b_{3} dx, \\
\RomanVIII_{k,4} \triangleq& \int_{\mathbb{R}^{3}} \partial_{k} b_{1}\partial_{2} b_{1} \partial_{k} \partial_{3} b_{1} dx, \\
\RomanVIII_{k,5} \triangleq& - \int_{\mathbb{R}^{3}}\partial_{k} b_{1} \partial_{1} b_{3}\partial_{k}\partial_{1} b_{2} dx, \\
\RomanVIII_{k,6} \triangleq& \int_{\mathbb{R}^{3}} \partial_{k} b_{1}\partial_{1}b_{3}\partial_{k}\partial_{2} b_{1} dx, \\
\RomanVIII_{k,7} \triangleq&\int_{\mathbb{R}^{3}} \partial_{k} b_{1} \partial_{3} b_{1}\partial_{k}\partial_{1} b_{2} dx, \\
\RomanVIII_{k,8} \triangleq& -\int_{\mathbb{R}^{3}} \partial_{k} b_{1}\partial_{3} b_{1}\partial_{k}\partial_{2} b_{1} dx. 
\end{align}
\end{subequations}
We observe that $\RomanVIII_{k,4} + \RomanVIII_{k,8}$ vanishes together as follows: 
\begin{align}
\RomanVIII_{k,4} + \RomanVIII_{k,8} =& \int_{\mathbb{R}^{3}} \partial_{k} b_{1} \partial_{2} b_{1} \partial_{k} \partial_{3} b_{1} - \partial_{k} b_{1}\partial_{3} b_{1} \partial_{k} \partial_{2} b_{1} dx \nonumber \\
=& \int_{\mathbb{R}^{3}} \partial_{2} b_{1} \frac{1}{2} \partial_{3} (\partial_{k} b_{1})^{2} - \partial_{3} b_{1} \frac{1}{2}\partial_{2} (\partial_{k} b_{1})^{2} dx\nonumber \\
=& \frac{1}{2} \int_{\mathbb{R}^{3}} -\partial_{2}\partial_{3} b_{1}(\partial_{k} b_{1})^{2} + \partial_{2}\partial_{3} b_{1} (\partial_{k} b_{1})^{2} dx = 0. \label{new3}
\end{align}
Therefore, we have shown that 
\begin{align}
\RomanIV_{5} \overset{\eqref{est 87}}{=}& \sum_{k=1}^{3} \epsilon \int_{\mathbb{R}^{3}}\nabla \times (j\times b) \cdot \partial_{k}^{2} b dx \overset{\eqref{est 88}}{=} \epsilon \sum_{k=1}^{3}\sum_{i=1}^{6} \RomanV_{k,i} \nonumber \\
\overset{\eqref{est 90}\eqref{est 92}\eqref{est 94}}{=}& \epsilon \sum_{k=1}^{3}\sum_{l=1}^{8} \RomanVI_{k,l} + \RomanVII_{k,l} + \RomanVIII_{k,l} \nonumber \\
\overset{\eqref{new1} \eqref{new2} \eqref{new3}}{=}& \epsilon \sum_{k=1}^{3} ( \sum_{l \in \{2,3,4,6,7,8\}} \RomanVI_{k,l} + \sum_{l\in \{1,3,4,5,6,8 \}} \RomanVII_{k,l} + \sum_{l \in \{1,2,3,5,6,7\}} \RomanVIII_{k,l}) \label{est 99}
\end{align}
We wish to bound all these terms by a constant multiples of $\int_{\mathbb{R}^{3}}\lvert \nabla b \rvert \lvert \nabla b_{h} \rvert \lvert \nabla^{2} b_{h} \rvert dx$ similarly to \eqref{est 73}. We start with $ \sum_{l \in \{2,3,4,6,7,8\}} \RomanVI_{k,l} $ which actually presents the most difficulty. First, 
\begin{align}
\sum_{k=1}^{3} \RomanVI_{k,4} + \RomanVI_{k,8} \overset{\eqref{est 91}}{=}& \sum_{k=1}^{3} \int_{\mathbb{R}^{3}} \partial_{k} b_{3} \partial_{3} b_{1} \partial_{k}\partial_{3} b_{2} - \partial_{k}b_{3}  \partial_{3} b_{2}\partial_{k}\partial_{3} b_{1} dx\nonumber\\
\lesssim& \int_{\mathbb{R}^{3}} \lvert \nabla b \rvert \lvert \nabla b_{h} \rvert \lvert \nabla^{2} b_{h} \rvert dx. \label{est 96}
\end{align} 
Second, we integrate by parts and rely on divergence-free property so that $\partial_{3}b_{3} = -\partial_{1}b_{1} - \partial_{2}b_{2}$ to deduce 
\begin{align}
& \sum_{k=1}^{3} \RomanVI_{k,2} + \RomanVI_{k,6} \nonumber\\
\overset{\eqref{est 91}}{=}& \sum_{k=1}^{3}\int_{\mathbb{R}^{3}} - \partial_{k}b_{3}\partial_{1} b_{3} \partial_{k} \partial_{3} b_{2} + \partial_{k} b_{3} \partial_{2} b_{3}\partial_{k}\partial_{3} b_{1} dx \nonumber \\
=& \sum_{k=1}^{3}\int_{\mathbb{R}^{3}} (\partial_{k}\partial_{3}b_{3} \partial_{1}b_{3} + \partial_{k}b_{3}\partial_{1}\partial_{3}b_{3}) \partial_{k} b_{2} - (\partial_{k}\partial_{3} b_{3} \partial_{2}b_{3} + \partial_{k}b_{3}\partial_{2}\partial_{3} b_{3}) \partial_{k} b_{1} dx \nonumber\\
\lesssim& \int_{\mathbb{R}^{3}} \lvert \nabla b_{h} \rvert \lvert \nabla b \rvert \lvert \nabla^{2} b_{h} \rvert dx. \label{est 97}
\end{align}
Finally, we integrate by parts and rely on $\partial_{3}b_{3} = -\partial_{1}b_{1} - \partial_{2}b_{2}$ again to deduce 
\begin{align}
& \sum_{k=1}^{3} \RomanVI_{k,3} + \RomanVI_{k,7} \nonumber \\
\overset{\eqref{est 91}}{=}& \sum_{k=1}^{3} \int_{\mathbb{R}^{3}} - \partial_{k}b_{3}\partial_{3}b_{1}\partial_{k}\partial_{2}b_{3} + \partial_{k}b_{3}\partial_{3}b_{2}\partial_{k}\partial_{1}b_{3} dx \nonumber \\
=& \sum_{k=1}^{3}\int_{\mathbb{R}^{3}} - \partial_{3}b_{1}\frac{1}{2}\partial_{2} (\partial_{k}b_{3})^{2} + \partial_{3}b_{2}\frac{1}{2}\partial_{1}(\partial_{k}b_{3})^{2} dx \nonumber \\
=& \sum_{k=1}^{3} \frac{1}{2} \int_{\mathbb{R}^{3}} - \partial_{2} b_{1} \partial_{3} (\partial_{k} b_{3})^{2} + \partial_{1} b_{2} \partial_{3} (\partial_{k} b_{3})^{2} dx \nonumber \\
=& \sum_{k=1}^{3}\int_{\mathbb{R}^{3}} - \partial_{2}b_{1}\partial_{k}b_{3}\partial_{k}\partial_{3} b_{3} + \partial_{1}b_{2} \partial_{k} b_{3} \partial_{k} \partial_{3} b_{3} dx \lesssim \int_{\mathbb{R}^{3}} \lvert \nabla b \rvert \lvert \nabla b_{h} \rvert \lvert \nabla^{2} b_{h} \rvert dx.  \label{est 98}
\end{align}
The rest of the terms in the right hand side of \eqref{est 99} are easier to bound as follows:
\begin{align}
& \epsilon \sum_{k=1}^{3} (\sum_{l \in \{1,3,4,5,6,8 \}} \RomanVII_{k,l} + \sum_{l\in \{1,2,3,5,6,7 \}} \RomanVIII_{k,l}) \label{est 100}\\
\overset{\eqref{est 93}\eqref{est 95}}{=}& \epsilon \sum_{k=1}^{3} (\int_{\mathbb{R}^{3}} \partial_{k} b_{2} (\partial_{1} b_{2} \partial_{k} \partial_{2} b_{3} - \partial_{2} b_{1} \partial_{k} \partial_{2} b_{3} + \partial_{2} b_{1} \partial_{k}\partial_{3} b_{2} \nonumber \\
& \hspace{10mm} - \partial_{2} b_{3} \partial_{k}\partial_{1} b_{2} + \partial_{2} b_{3}\partial_{k}\partial_{2} b_{1} - \partial_{3} b_{2}\partial_{k}\partial_{2} b_{1}) dx \nonumber \\
& + \int_{\mathbb{R}^{3}}\partial_{k}b_{1} (\partial_{1} b_{2} \partial_{k}\partial_{1} b_{3} - \partial_{1}b_{2}\partial_{k}\partial_{3} b_{1} - \partial_{2} b_{1}\partial_{k}\partial_{1} b_{3} \nonumber \\
& \hspace{10mm} - \partial_{1} b_{3} \partial_{k} \partial_{1} b_{2} + \partial_{1} b_{3} \partial_{k}\partial_{2} b_{1} + \partial_{3} b_{1} \partial_{k}\partial_{1} b_{2} ) dx   \lesssim \int_{\mathbb{R}^{3}} \lvert \nabla b \rvert \lvert \nabla b_{h} \rvert \lvert \nabla^{2} b_{h} \rvert dx\nonumber 
\end{align}
where we integrated by parts on the four terms that have second partial derivatives on $b_{3}$, namely $\partial_{k}\partial_{2}b_{3}$ and $\partial_{k}\partial_{1}b_{3}$. We apply \eqref{est 96}, \eqref{est 97}, \eqref{est 98}, and \eqref{est 100} to \eqref{est 99} to conclude by the Sobolev embedding of $\dot{H}^{\frac{3}{4}}(\mathbb{R}^{3}) \hookrightarrow L^{4}(\mathbb{R}^{3})$ and Young's inequality that 
\begin{align}
\RomanIV_{5} \lesssim& \int_{\mathbb{R}^{3}} \lvert \nabla b \rvert \lvert \nabla b_{h} \rvert \lvert \nabla^{2} b_{h} \rvert dx \lesssim \lVert \nabla b \rVert_{L^{2}} \lVert \nabla b_{h} \rVert_{L^{4}} \lVert \nabla^{2} b_{h} \rVert_{L^{4}}  \nonumber\\
\lesssim& \lVert \nabla b \rVert_{L^{2}} \lVert \Lambda^{\frac{7}{4}} b_{h} \rVert_{L^{2}}\lVert \Lambda^{\frac{11}{4}} b_{h} \rVert_{L^{2}} \leq \frac{\eta_{h}}{2} \lVert \Lambda^{\frac{11}{4}} b_{h} \rVert_{L^{2}}^{2} + C\lVert \nabla b\rVert_{L^{2}}^{2} \lVert \Lambda^{\frac{7}{4}} b_{h} \rVert_{L^{2}}^{2}. \label{est 101}
\end{align} 
Applying \eqref{est 86} and \eqref{est 101} to \eqref{est 82} gives us 
\begin{align}
& \partial_{t} (\lVert \nabla u \rVert_{L^{2}}^{2} + \lVert \nabla b\rVert_{L^{2}}^{2}) + \nu \lVert \Lambda^{\frac{9}{4}} u \rVert_{L^{2}}^{2} + \eta_{h} \lVert \Lambda^{\frac{11}{4}} b_{h} \rVert_{L^{2}}^{2} + \eta_{v} \lVert \Lambda^{\frac{9}{4}} b_{v} \rVert_{L^{2}}^{2}  \label{est 102} \\
\lesssim& (1+ \lVert \nabla u \rVert_{L^{2}}^{2} + \lVert \nabla b \rVert_{L^{2}}^{2})(\lVert \Lambda^{\frac{5}{4}} u \rVert_{L^{2}}^{2} + \lVert \Lambda^{\frac{7}{4}} b_{h} \rVert_{L^{2}}^{2} + \lVert \Lambda^{\frac{5}{4}} b_{v} \rVert_{L^{2}}^{2}). \nonumber 
\end{align}
Applying Gronwall's inequality on \eqref{est 102} and relying on \eqref{est 81} allow us to deduce in particular that $\int_{0}^{T} \lVert \Lambda^{\frac{9}{4}} b \rVert_{L^{2}}^{2} dt < \infty$; applying this bound to \eqref{est 85} completes the $H^{m}(\mathbb{R}^{3})$-bound of $(u,b)$ and hence the proof of Theorem \ref{Theorem 2.5}. 

\section{Appendix: Proof of Theorem \ref{Theorem 2.3} (2)} 
As we mentioned, Propositions \ref{Proposition 3.1}-\ref{Proposition 3.2} remain valid for us. We obtain the following proposition instead of Proposition \ref{Proposition 3.3}: 
 
\begin{proposition}\label{Proposition 4.1}
Under the hypothesis of Theorem \ref{Theorem 2.3} (2), let $\nabla \times \omega $ be a smooth solution to \eqref{est 37} over $[0,T]$. Then for both $k \in \{1,2\}$, 
\begin{equation}\label{est 60} 
(\nabla \times \omega)_{k} \in L_{T}^{\infty} L_{x}^{2} \cap L_{T}^{2} \dot{H}_{x}^{1}. 
\end{equation} 
\end{proposition}

\begin{proof}[Proof of Proposition \ref{Proposition 4.1}] 
For clarity, let us prove Theorem \ref{Theorem 2.3} (2) when $\Delta (\Delta u_{1}) \in L_{T}^{r_{1}} L_{x}^{p_{1}}$ with $\frac{2}{p_{1}} + \frac{2}{r_{1}} \leq 1, p_{1} \in (2,\infty)$ and $\Delta (\Delta u_{2}) \in L_{T}^{2} BMO_{x}$ as the other cases can be proven similarly. We take $L^{2}(\mathbb{R}^{2})$-inner products on \eqref{est 37} with 
$\begin{pmatrix}
(\nabla \times \omega)_{1} & (\nabla \times \omega)_{2} & 0
\end{pmatrix}^{T}$ to deduce 
\begin{equation}\label{est 56}
\sum_{k=1}^{2} \frac{1}{2} \partial_{t} \lVert (\nabla \times \omega)_{k} \rVert_{L^{2}}^{2} + \lVert \nabla (\nabla \times \omega)_{k} \rVert_{L^{2}}^{2} = \sum_{k=1}^{2} \sum_{i=1}^{3} \RomanII_{k,i} 
\end{equation}  
where 
\begin{subequations}\label{est 55}
\begin{align}
& \RomanII_{k,1} \triangleq - \int_{\mathbb{R}^{2}} (\nabla \times ((u\cdot\nabla) \omega))_{k} (\nabla \times \omega)_{k} dx, \\
& \RomanII_{k,2} \triangleq \int_{\mathbb{R}^{2}} (\nabla \times ((\omega \cdot \nabla) u))_{k} (\nabla \times \omega)_{k} dx, \\
& \RomanII_{k,3} \triangleq \int_{\mathbb{R}^{2}} (\nabla \times \nabla \times ((b\cdot\nabla) b))_{k} (\nabla \times \omega)_{k} dx. 
\end{align}
\end{subequations} 
Using \eqref{est 41}-\eqref{est 42} we can rewrite using the fact that $\nabla \times \omega = - \Delta u$ 
\begin{subequations}
\begin{align}
&\sum_{k=1}^{2} \RomanII_{k,1}  =  \int_{\mathbb{R}^{2}} (u\cdot\nabla) \omega_{3} \Delta \omega_{3} dx, \\
& \sum_{k=1}^{2} \RomanII_{k,2} = - \int_{\mathbb{R}^{2}} (\omega \cdot \nabla) u_{3} \Delta \omega_{3} dx, \\
&\sum_{k=1}^{2} \RomanII_{k,3} = -\int_{\mathbb{R}^{2}} 
\begin{pmatrix}
(b\cdot\nabla) b_{1} & (b\cdot\nabla) b_{2} & 0
\end{pmatrix}^{T}
\cdot 
\begin{pmatrix} 
\Delta (\nabla \times \omega)_{1} & \Delta ( \nabla \times \omega)_{2} & 0 
\end{pmatrix}^{T} dx. 
\end{align}
\end{subequations} 
We can estimate $\sum_{k=1}^{2} \RomanII_{k,1}$ and $\sum_{k=1}^{2} \RomanII_{k,2}$ by H$\ddot{\mathrm{o}}$lder's, Gagliardo-Nirenberg, and Young's inequalities, 
\begin{align}
\sum_{k=1}^{2} \RomanII_{k,1} \lesssim& \lVert u \rVert_{L^{\infty}} \lVert \nabla \omega_{3} \rVert_{L^{2}} \lVert \Delta (\partial_{1} u_{2} - \partial_{2} u_{1}) \rVert_{L^{2}} \lesssim \sum_{k,l=1}^{2} \lVert u \rVert_{L^{2}}^{\frac{1}{2}} \lVert \Delta u \rVert_{L^{2}}^{\frac{1}{2}} \lVert \Delta u_{k} \rVert_{L^{2}} \lVert \Delta \nabla u_{l} \rVert_{L^{2}} \nonumber \\
\leq& \frac{1}{4} \sum_{k=1}^{2} \lVert \nabla (\nabla \times \omega)_{k} \rVert_{L^{2}}^{2} + C \lVert u \rVert_{L^{2}} \lVert \Delta u \rVert_{L^{2}} \sum_{k=1}^{2} \lVert (\nabla \times \omega)_{k} \rVert_{L^{2}}^{2}, \label{est 57}
\end{align}
\begin{align}
\sum_{k=1}^{2} \RomanII_{k,2} \lesssim& \lVert \omega \rVert_{L^{4}} \lVert \nabla u_{3} \rVert_{L^{4}} \lVert \Delta \omega_{3} \rVert_{L^{2}} \lesssim \lVert \nabla u \rVert_{L^{2}} \lVert \Delta u \rVert_{L^{2}} \sum_{k=1}^{2} \lVert \Delta \nabla u_{k} \rVert_{L^{2}} \nonumber \\
\leq& \frac{1}{4} \sum_{k=1}^{2} \lVert \nabla (\nabla \times \omega)_{k} \rVert_{L^{2}}^{2}+ C \lVert \nabla u \rVert_{L^{2}}^{2} \lVert \Delta u \rVert_{L^{2}}^{2}. \label{est 58}
\end{align}
For $p_{1} \in (2,\infty)$ we estimate via H$\ddot{\mathrm{o}}$lder's inequality, duality of $\mathcal{H}^{1}(\mathbb{R}^{2})$ and $BMO(\mathbb{R}^{2})$, Gagliardo-Nirenberg inequality, and the div-curl lemma, 
\begin{align}
\sum_{k=1}^{2} \RomanII_{k,3} \lesssim& \lVert b \rVert_{L^{\frac{2p_{1}}{p_{1} -2}}} \lVert \nabla b_{1} \rVert_{L^{2}} \lVert \Delta (\nabla \times \omega)_{1} \rVert_{L^{p_{1}}} + \lVert (b\cdot\nabla) b_{2} \rVert_{\mathcal{H}^{1}} \lVert \Delta (\nabla \times \omega)_{2} \rVert_{BMO} \nonumber \\
\lesssim& \lVert b \rVert_{L^{2}}^{\frac{p_{1} -2}{p_{1}}} \lVert \nabla b \rVert_{L^{2}}^{\frac{2+ p_{1}}{p_{1}}} \lVert \Delta (\Delta u_{1}) \rVert_{L^{p_{1}}} + \lVert b \rVert_{L^{2}} \lVert \nabla b\rVert_{L^{2}} \lVert \Delta (\Delta u_{2} )\rVert_{BMO}. \label{est 59} 
\end{align}
Thus, we can now apply \eqref{est 57}, \eqref{est 58}, and \eqref{est 59} to \eqref{est 56}, integrate over $[0,t]$, and use H$\ddot{\mathrm{o}}$lder's inequality to deduce 
\begin{align}
& \sum_{k=1}^{2} \lVert (\nabla \times \omega)_{k} \rVert_{L^{2}}^{2} + \int_{0}^{t} \lVert \nabla (\nabla \times \omega)_{k} \rVert_{L^{2}}^{2} ds \\
\overset{\eqref{est 30}\eqref{est 46}}{\leq}&\sum_{k=1}^{2}  \lVert (\nabla \times \omega)_{k}(0) \rVert_{L^{2}}^{2} + C( \int_{0}^{t} \sum_{k=1}^{2} \lVert (\nabla \times \omega)_{k} \rVert_{L^{2}}^{2} \lVert \Delta u \rVert_{L^{2}} ds + 1 \nonumber\\
& \hspace{15mm} +  \left(\int_{0}^{t} \lVert \Delta (\Delta u_{1}) \rVert_{L^{p_{1}}}^{\frac{2p_{1}}{p_{1} -2}} ds \right)^{\frac{p_{1} -2}{2p_{1}}} + \left(\int_{0}^{t} \lVert \Delta (\Delta u_{2})\rVert_{BMO}^{2} ds \right)^{\frac{1}{2}} ) \nonumber 
\end{align}
from which Gronwall's inequality completes the proof of Proposition \ref{Proposition 4.1}. 
\end{proof}

By Propositions \ref{Proposition 3.2} and \ref{Proposition 4.1} we realize that for both $k \in \{1,2\}$ 
\begin{equation}\label{est 76}
j_{k} \overset{\eqref{est 32}}{=} z_{k}^{2} - (\nabla \times \omega)_{k} \in L_{T}^{\infty} L_{x}^{2} \cap L_{T}^{2} \dot{H}_{x}^{1}. 
\end{equation} 
We can apply Gagliardo-Nirenberg inequality to deduce from  \eqref{est 76} that for both $k \in \{1,2\}$ 
\begin{equation}
\int_{0}^{T} \lVert j_{k} \rVert_{L^{4}}^{4} ds \lesssim \sup_{s\in [0,t]} \lVert j_{k} (s) \rVert_{L^{2}}^{2} \int_{0}^{T} \lVert \nabla j_{k}  \rVert_{L^{2}}^{2} ds \lesssim 1. 
\end{equation} 
Therefore, $j_{k} \in L_{T}^{4}L_{x}^{4}$ for both $k \in \{1,2\}$, allowing us to apply Theorem \ref{Theorem 2.1} to deduce \eqref{est 27} as desired. The proof of Theorem \ref{Theorem 2.3} (2) is now complete. 
 
 \section{Appendix: Proof of Proposition \ref{Proposition 3.4}} 
We apply $D^{\alpha}$ on \eqref{est 78} with $\alpha$ being any multi-index such that $\lvert \alpha \rvert \leq m$, take $L^{2}(\mathbb{R}^{3})$-inner products with $(D^{\alpha}u, D^{\alpha} b)$, sum over all such $\alpha$ and obtain an identity of 
\begin{align}
\frac{1}{2} \partial_{t} (\lVert u \rVert_{H^{m}}^{2} + \lVert b \rVert_{H^{m}}^{2}) + \nu \lVert \Lambda^{\frac{5}{4}} u \rVert_{H^{m}}^{2} + \eta_{h} \lVert \Lambda^{\frac{7}{4}} b_{h} \rVert_{H^{m}}^{2} + \eta_{v} \lVert \Lambda^{\frac{5}{4}} b_{v} \rVert_{H^{m}}^{2} = \sum_{i=1}^{4} \RomanIII_{i} \label{est 84} 
\end{align}
where 
\begin{subequations}
\begin{align}
& \RomanIII_{1} \triangleq -\sum_{\lvert \alpha \rvert \leq m} \int_{\mathbb{R}^{3}} D^{\alpha} [(u\cdot\nabla) u]\cdot D^{\alpha} u dx, \\
& \RomanIII_{2} \triangleq -\sum_{\lvert \alpha \rvert \leq m} \int_{\mathbb{R}^{3}} D^{\alpha} [(u\cdot\nabla) b] \cdot D^{\alpha} b dx,\\
& \RomanIII_{3} \triangleq \sum_{\lvert \alpha \rvert \leq m} \int_{\mathbb{R}^{3}} D^{\alpha} [( b\cdot\nabla) b] \cdot D^{\alpha} u + D^{\alpha} [ (b\cdot\nabla) u] \cdot D^{\alpha} b dx,\\
& \RomanIII_{4} \triangleq - \epsilon \sum_{\lvert \alpha \rvert \leq m} \int_{\mathbb{R}^{3}} D^{\alpha} [ ( j \times b) ] \cdot D^{\alpha} j dx.
\end{align}
\end{subequations}
Making use of the well-known identities, 
\begin{align*}
& \int_{\mathbb{R}^{3}} (u\cdot\nabla) D^{\alpha} u \cdot D^{\alpha} u dx = 0, \hspace{5mm}  \int_{\mathbb{R}^{3}} (u\cdot\nabla) D^{\alpha} b \cdot D^{\alpha} b dx = 0, \\
& \int_{\mathbb{R}^{3}} (b\cdot\nabla) D^{\alpha} b \cdot D^{\alpha} u + (b\cdot\nabla) D^{\alpha} u  \cdot D^{\alpha} b dx = 0, \\
& \int_{\mathbb{R}^{3}} D^{\alpha} j \times b \cdot D^{\alpha} j dx \overset{\eqref{key}}{=}0, 
\end{align*}
we can estimate by using H$\ddot{\mathrm{o}}$lder's inequality, the Sobolev embeddings $\dot{H}^{\frac{1}{4}} (\mathbb{R}^{3}) \hookrightarrow L^{\frac{12}{5}}(\mathbb{R}^{3})$ and $\dot{H}^{\frac{5}{4}} (\mathbb{R}^{3}) \hookrightarrow L^{12}(\mathbb{R}^{3})$, and Gagliardo-Nirenberg inequality 
\begin{subequations}\label{est 83} 
\begin{align}
&  \RomanIII_{1}  \lesssim \lVert u \rVert_{H^{m}} \lVert \Lambda^{\frac{5}{4}} u \rVert_{L^{2}} \lVert \Lambda^{\frac{5}{4}} u \rVert_{H^{m}},  \\
&  \RomanIII_{2} \lesssim (\lVert \Lambda^{\frac{5}{4}} u \rVert_{H^{m}} \lVert \Lambda^{\frac{5}{4}} b\rVert_{L^{2}} + \lVert \Lambda^{\frac{5}{4}} u \rVert_{L^{2}} \lVert \Lambda^{\frac{5}{4}} b \rVert_{H^{m}}) \lVert b \rVert_{H^{m}}, \\
&  \RomanIII_{3} \lesssim (\lVert u \rVert_{H^{m}} + \lVert b \rVert_{H^{m}}) (\lVert \Lambda^{\frac{5}{4}} u \rVert_{H^{m}} + \lVert \Lambda^{\frac{5}{4}} b \rVert_{H^{m}})(\lVert \Lambda^{\frac{5}{4}} u \rVert_{L^{2}} + \lVert \Lambda^{\frac{5}{4}} b \rVert_{L^{2}}), \\
& \RomanIII_{4}\lesssim  \lVert \Lambda^{\frac{9}{4}} b \rVert_{L^{2}} \lVert b \rVert_{H^{m}} \lVert \Lambda^{\frac{5}{4}} b \rVert_{H^{m}}.
\end{align}
\end{subequations} 
Applying \eqref{est 83} to \eqref{est 84} and relying on Young's inequality give us for any $\delta > 0$, 
\begin{align}
&\frac{1}{2} \partial_{t} (\lVert u \rVert_{H^{m}}^{2} + \lVert b \rVert_{H^{m}}^{2}) + \nu \lVert \Lambda^{\frac{5}{4}} u \rVert_{H^{m}}^{2} + \eta_{h} \lVert \Lambda^{\frac{7}{4}} b_{h} \rVert_{H^{m}}^{2} + \eta_{v} \lVert \Lambda^{\frac{5}{4}} b_{v} \rVert_{H^{m}}^{2}  \nonumber \\
\leq& \frac{\nu}{2} \lVert \Lambda^{\frac{5}{4}} u \rVert_{H^{m}}^{2} + \delta \lVert \Lambda^{\frac{5}{4}} b \rVert_{H^{m}}^{2} + C(\lVert u \rVert_{H^{m}}^{2} + \lVert b \rVert_{H^{m}}^{2}) ( \lVert \Lambda^{\frac{5}{4}} u \rVert_{L^{2}}^{2} + \lVert \Lambda^{\frac{5}{4}} b \rVert_{L^{2}}^{2} + \lVert \Lambda^{\frac{9}{4}} b \rVert_{L^{2}}^{2})  \label{est 103} 
\end{align} 
where Plancherel theorem and Young's inequality allow us to write 
\begin{align*}
\delta \lVert \Lambda^{\frac{5}{4}} b \rVert_{H^{m}}^{2} \leq 2 \delta ( \lVert \Lambda^{\frac{5}{4}} b_{h} \rVert_{H^{m}}^{2} + \lVert \Lambda^{\frac{5}{4}} b_{v} \rVert_{H^{m}}^{2}) \leq 2 \delta ( C ( \lVert b_{h} \rVert_{L^{2}}^{2} + \lVert \Lambda^{\frac{7}{4}} b_{h} \rVert_{H^{m}}^{2}) + \lVert \Lambda^{\frac{5}{4}} b_{v} \rVert_{H^{m}}^{2}),
\end{align*}
as well as $ \lVert \Lambda^{\frac{5}{4}} b \rVert_{L^{2}}^{2} \lesssim \lVert b \rVert_{L^{2}}^{2} + \lVert \Lambda^{\frac{9}{4}} b \rVert_{L^{2}}^{2}$. Thus, taking $\delta > 0$ sufficiently small and applying Gronwall's inequality on \eqref{est 103} lead us to \eqref{est 85} and complete the proof of Proposition \ref{Proposition 3.4}. 

\section*{Acknowledgements}
The authors express deep gratitude to Prof. Giorgio Bornia and Prof. Mimi Dai for valuable discussions.  

\section*{Statements and Declarations}
The authors declare that they have no conflict of interest.

\end{document}